\RequirePackage[l2tabu,orthodox]{nag}
\documentclass[a4paper,11pt]{article}

\newcommand{\mytitle}{
  Rate Constant Matrix Contraction Method for\\
  Stiff Master Equations with Detailed Balance%
}
\title{\mytitle}

\usepackage[margin=1truein]{geometry}

\usepackage{lmodern}
\usepackage[T1]{fontenc}
\usepackage{textcomp}

\usepackage{amsmath}
\usepackage{amssymb}
\usepackage{amsthm}

\newtheorem{theorem}{Theorem}[section]
\newtheorem{lemma}[theorem]{Lemma}
\newtheorem{proposition}[theorem]{Proposition}

\theoremstyle{definition}

\usepackage{algorithm}
\usepackage[noend]{algpseudocode}

\algnewcommand{\algorithmicinput}{\textbf{Input:}}
\algnewcommand{\Input}{\item[\algorithmicinput]}
\algnewcommand{\algorithmicoutput}{\textbf{Output:}}
\algnewcommand{\Output}{\item[\algorithmicoutput]}

\usepackage{mathtools}
\usepackage{nicematrix}
\usepackage{siunitx}

\usepackage{multirow}
\usepackage{booktabs}

\usepackage{upref}
\usepackage[bookmarksnumbered,hypertexnames=false,pdfdisplaydoctitle,pdfusetitle,unicode]{hyperref}
\usepackage{cleveref}

\Crefname{step}{Step}{Steps}
\Crefname{step}{Step}{Steps}

\newcounter{type}
\renewcommand\thetype{\Alph{type}}

\crefname{type}{Type}{Types}

\expandafter\def\csname ver@etex.sty\endcsname{3000/12/31}
\usepackage{autonum}
\makeatletter
\autonum@generatePatchedReferenceCSL{Cref}
\makeatother
\newcommand{\retainlabel}[1]{\label{#1}\sbox0{\ref{#1}}}

\usepackage{etoolbox}
\cslet{blx@noerroretextools}\empty%
\usepackage[date=year,isbn=false,maxcitenames=2,natbib,url=false,sorting=nyt,style=trad-abbrv]{biblatex}
\usepackage[color={red!100!green!33},textsize=scriptsize]{todonotes}

\usepackage{okicmd}
\usepackage[noalgorithm,notheorem]{okithm}

\newcommand{\zeros}{\mathbf{0}}
\newcommand{\zerostr}{\trsp{\zeros}}
\renewcommand{\ones}{\mathbf{1}}
\newcommand{\onestr}{\trsp{\ones}}
\newcommand{\pipr}[1]{{\agbr{#1}}_\pi}
\newcommand{\piprbig}[1]{{\agbr[\big]{#1}}_\pi}
\newcommand{\pinorm}[1]{{\norm{#1}}_\pi}
\newcommand{\pinormbig}[1]{{\norm[\big]{#1}}_\pi}
\newcommand{\given}{\mathbin{|}}
\newcommand{\E}{\mathbb{E}}
\newcommand{\Err}{\mathrm{Err}}
\def\CC{{C\nolinebreak[4]\hspace{-.05em}\raisebox{.4ex}{\tiny\textbf{++}}}}

\DeclareMathOperator{\proj}{proj}
\DeclareMathOperator{\nnz}{nnz}

\author{
  Satoru Iwata%
  \thanks{
    Department of Mathematical Informatics, Graduate School of Information Science and Technology, University of Tokyo, Tokyo 113-8656, Japan.
    E-mail: \{%
      \href{mailto:iwata@mist.i.u-tokyo.ac.jp}{\texttt{\nolinkurl{iwata}}},
      \href{mailto:oki@mist.i.u-tokyo.ac.jp}{\texttt{\nolinkurl{oki}}},
      \href{mailto:sakaue@mist.i.u-tokyo.ac.jp}{\texttt{\nolinkurl{sakaue}}}%
    \}\texttt{\nolinkurl{@mist.i.u-tokyo.ac.jp}}
  }
  \thanks{Institute for Chemical Reaction Design and Discovery (ICReDD), Hokkaido University, Sapporo, Hokkaido 001-0021, Japan.}
  \and
  Taihei Oki\footnotemark[1]
  \and
  Shinsaku Sakaue\footnotemark[1]
}

\ifpdf
\hypersetup{%
  pdftitle={Rate Constant Matrix Contraction Method for Stiff Master Equations with Detailed Balance},%
  pdfauthor={Satoru Iwata, Taihei Oki, and Shinsaku Sakaue},%
  pdfkeywords={Chemical kinetics, differential equations, quasi-steady-state approximation, rate constant matrix contraction method}%
}
\fi

\begin{document}
\maketitle

\begin{abstract}
This paper considers master equations for Markovian kinetic schemes that possess the detailed balance property.
Chemical kinetics, as a prime example, often yields large-scale, highly stiff equations.
Based on chemical intuitions, Sumiya et al.~(2015) presented the rate constant matrix contraction (RCMC) method that computes approximate solutions to such intractable systems.

This paper aims to establish a mathematical foundation for the RCMC method.
We present a reformulated RCMC method in terms of matrix computation, deriving the method from several natural requirements.
We then perform a theoretical error analysis based on eigendecomposition and discuss implementation details caring about computational efficiency and numerical stability.
Through numerical experiments on synthetic and real kinetic models, we validate the efficiency, numerical stability, and accuracy of the presented method.
\end{abstract}

\section{Introduction}
A \emph{master equation} is a system of first-order ordinary differential equations (ODEs) describing the time evolution of the quantity or probability distribution on a finite number of states.
This paper focuses on master equations for \emph{Markovian kinetic schemes}, which are linear ODEs written in the form of
\begin{align}\label{def:master}
  \dot{x}(t) = Kx(t),
\end{align}
where $x(t)$ is an $n$-dimensional vector whose $i$th component $x_i(t)$ represents the quantity or probability of occupying state $i$ at time $t \in \R$ and $K$ is an $n \times n$ matrix whose $(i, j)$ entry $K_{ij}$ with $i \ne j$ is a non-negative value called the (\emph{transition}) \emph{rate constant} from state $j$ to state $i$.
Each diagonal entry in $K$ is the minus of the sum of off-diagonals in the column to which the diagonal belongs, i.e., $K_{jj} = -\sum_{i \ne j} K_{ij}$.
This equality implies that the equation satisfies the \emph{conservation law}, i.e., $\sum_{i=1}^n x_i(t)$ is preserved throughout time evolution.
The matrix $K$ is called a \emph{rate constant matrix}.

A master equation~\eqref{def:master} is said to satisfy the \emph{detailed balance condition} if there exists a positive vector $\pi = \prn{\pi_i}_{n} \in \R^n$ such that
\begin{align}\label{def:detailed-balance}
  K_{ij} \pi_j = K_{ji} \pi_i
\end{align}
holds for each pair of $i$ and $j$.
The condition~\eqref{def:detailed-balance} implies $K \pi = \zeros$.
If the system is irreducible (i.e., the supporting graph of $K$ is (strongly) connected), $\pi$ is a \emph{stationary} (or \emph{equilibrium}) \emph{distribution}, which is a unique distribution such that $x(t)$ converges to $\pi$ as $t \to \infty$ regardless of the initial distribution $p = {(p_i)}_{i} = x(0)$ with $\sum_{i=1}^n p_i = \sum_{i=1}^n \pi_i$, as expounded in~\cite{Anderson1991-ec}.

Master equations with detailed balance often arise from a monomolecular chemical kinetics simulation of reaction path networks.
Within this context, the states represent energetically stable atomic configurations in the potential energy surface, called \emph{equilibrium states} (EQs).
If the system is described by a canonical ensemble, the rate constant $K_{ij}$ from state $j$ to state $i$ with $i \ne j$ is given by
\begin{align}\label{eq:canonical}
  K_{ij} = \Gamma \frac{k_{\mathrm{B}}T}{h} \exp\prn{-\frac{E_{\text{$i$--$j$}} - E_j}{RT}},
\end{align}
where $k_\mathrm{B}$ is the Boltzmann constant, $h$ is the Planck constant, $R$ is the gas constant, $T$ is the temperature, $\Gamma$ is the transmission coefficient, and $E_j$ and $E_{\text{$i$--$j$}} = E_{\text{$j$--$i$}}$ are the potential energies of equilibrium state $j$ and the transition state between equilibrium states $i$ and $j$, respectively~\cite{Sumiya2015-br}.
The rate constant matrix $K$ determined by~\eqref{eq:canonical} satisfies the detailed balance condition~\eqref{def:detailed-balance} with the stationary distribution $\pi$ given by
\begin{align}
  \pi_i = \frac{
    \exp\prn[\big]{-\frac{E_i}{RT}}
  }{
    \sum_{j=1}^n \exp\prn[\big]{-\frac{E_j}{RT}}
  },
\end{align}
which is called the \emph{Boltzmann distribution}.
Rate constant matrices derived from the microcanonical ensemble also satisfy the detailed balance condition; see~\cite{Sumiya2017-qo}.

Chemical kinetics simulation often yields stiff master equations due to a wide range of time scales of reactions.
While the analytic solution of~\eqref{def:master} can be readily expressed as $x(t) = \e^{tK}p$ for any rate constant matrix $K$ and initial distribution $p$, solving stiff ODEs numerically has been recognized as a difficult task.
A classical remedy for the stiffness issue in kinetics simulation is to apply \emph{quasi-steady-state approximation} (QSSA), introduced in~\cite{Bodenstein1913-tc}.
Given a subset $S \subseteq [n] \defeq \set{1, \dotsc, n}$ of states that are regarded to be ``steady,'' i.e., $\dot{x}_i(t) \approx 0$ for $i \in S$, the QSSA replaces the left-hand side $\dot{x}_i(t)$ of~\eqref{def:master} with $0$ for every $i \in S$.
In other words, the QSSA constrains the solution to a subspace determined from $S$, which we call the \emph{QSSA subspace}.
The approximate equation is solved as a differential-algebraic equation (DAE), or one can eliminate the algebraic variables $x_i \; (i \in S)$ to get a new master equation of smaller size.
While QSSA is conceptually clear and easy to compute, the conservation law is no longer valid in the approximate DAE.

In recent years, large-scale, highly stiff master equations have emerged with the development of automatic chemical reaction pathway search techniques such as the anharmonic downward distortion following (ADDF) method~\cite{Ohno2004-iw} and the single-component artificial force induced reaction (SC-AFIR) method~\cite{Maeda2014-cm}.
To address such intractable equations, Sumiya et al.~\cite{Sumiya2015-br,Sumiya2017-qo} presented the \emph{rate constant matrix contraction} (RCMC) method for systems described by the canonical or microcanonical ensemble.
Based on QSSA, the RCMC method maintains a set $S \subseteq [n]$ of states regarded to be steady, which monotonically grows from the empty set.
For each $S$, the method computes a vector $q = {(q_i)}_{i} \in \R^n$ that approximates the solution $x(t)$ at time $t$ when the states in $S$ and $T \defeq [n] \setminus S$ can be regarded to be steady and transient, respectively.
The output $q$ preserves the component sum, i.e., $\sum_{i=1}^n p_i = \sum_{i=1}^n q_i$ holds, and the method runs fast for large-scale, stiff equations because it does not rely on numerical integration, unlike ordinary numerical methods.
The RCMC method has been successfully applied to the kinetics-based navigation in the SC-AFIR method~\cite{Sumiya2020-ti}.
The RCMC method was, however, designed based on chemical intuition, and its mathematical validity and properties have not been discussed thoroughly.

\paragraph{Contributions.}
This paper improves the RCMC method by establishing mathematical foundations, summarized as follows.

First, we reformulate the RCMC method in terms of matrix computations.
The reformulated method applies to any master equation~\eqref{def:master} whose rate constant matrix $K$ satisfies the detailed balance condition~\eqref{def:detailed-balance}.
We describe the method by means of matrix operations, while all calculation formulas in the original form~\cite{Sumiya2015-br,Sumiya2017-qo} are written with direct reference to the matrix entries.
In our reformulation, the output $q$ of the method is expressed as a matrix-vector multiplication $Vp$, where $p \in \R^n$ is a given initial value and $V \in \R^{n \times n}$ is a matrix approximating $\e^{tK}$ that can be explicitly constructed from the rate constant matrix $K$ and the set $S$ of steady states.
We observe that the matrix $V$ can be determined from natural requirements coming from properties of $\e^{tK}$ and the QSSA.
Furthermore, we propose a variant of $V$, which is another natural approximation of $\e^{tK}$.
The methods with original and modified $V$ are referred to as Types~A and~B, respectively.
While the matrix $V$ of Type~A is easier to compute, the matrix $V$ of Type~B is mathematically more tractable as it is indeed the orthogonal projection onto the QSSA subspace.

Second, we theoretically analyze the error between the approximate solution $q = Vp$ and the exact solution $x(t) = \e^{tK}p$.
Here, the error is measured in terms of the \emph{$\pi$-norm}; see \Cref{sec:rate-constant-matrices} for definition.
The derived error bounds depend on the spectral radii of matrices obtained from $K$ and $S$, as well as the coefficients in the expansion of $p$ with respect to an eigenbasis of $K$.
Roughly speaking, $q$ becomes a good approximation to $x(t)$ if $p$ consists of eigenvectors of $K$ with eigenvalues far from the spectral radii.
We give a better and simpler bound for Type~B than for Type~A.
Our error estimation for Type~B leads to a principled way to determine the optimal reference time $t^*$ such that $q$ should be regarded as an approximation of $x(t^*)$.

Third, based on our reformulation and subsequent error analysis, we discuss implementation details of the RCMC method.
We provide efficient implementations for calculating the optimal reference time $t^*$ and for the incremental updates of the LU decomposition within the algorithm.
We also address concerns regarding numerical stability.

Finally, we conduct numerical experiments using one synthetic dataset and two actual chemical reaction networks to demonstrate the efficiency, numerical stability, and accuracy of the RCMC method.
We observe that the RCMC method of both Types~A and~B using double precision is capable of producing approximate solutions within mere seconds.
The experimental results substantiate that the actual errors are within our theoretical bounds, subject to issues arising from the machine epsilon.
Furthermore, we compare three methods for calculating the optimal reference time $t^*$ and analyze the errors using $\pi$- and $L_\infty$-norms.

\paragraph{Previous and Related Work.}
The RCMC method was originally introduced in~\cite{Sumiya2015-br} as a method to reduce the master equation~\eqref{def:master} with the rate constant matrix derived from the canonical ensemble~\eqref{eq:canonical} into a smaller system by combining repeated applications of QSSA and an operation similar to \emph{lumping}~\cite{Kuo1969-vh,Wei1969-kr}, which is a classical reduction method for (generally nonlinear) master equations by grouping several states into a single state.
The RCMC method in the aforementioned form of computing approximate solutions was subsequently proposed in~\cite{Sumiya2017-qo}, inheriting the spirit of lumping in its design.

QSSA is classified as a reduction method for generally nonlinear master equations based on timescale analysis.
Other representative examples of methods based on timescale analysis include \emph{intrinsic low-dimensional manifold} (ILDM) method~\cite{Maas1992-ft,Maas1992-xl} and \emph{computational singular perturbation} (CSP)~\cite{Lam1989-zp}.
Both methods rely on the eigendecomposition of the rate constant matrix (or the Jacobian matrix for nonlinear systems), which is prohibitive for large-scale, highly stiff systems we are concerned with.
In contrast, the RCMC method can be executed without performing the eigendecomposition, while our error analysis depends on it.
See~\cite[Chapter~7]{Turanyi2014-ky} for a comprehensive survey of reduction techniques for kinetics equations.

\paragraph{Organization.}
The rest of this paper is organized as follows.
\Cref{sec:preliminaries} affords preliminaries on rate constant matrices, master equations, QSSA, and the RCMC method.
\Cref{sec:reformulating-matrix-update-formulas} reformulates the RCMC method in terms of matrix operations.
\Cref{sec:error-analysis} provides an error analysis for the RCMC method.
\Cref{sec:improved-rcmc-method} presents the overall framework of the RCMC method and discusses its implementation details.
\Cref{sec:experiments} shows the results of numerical experiments.
Finally, \Cref{sec:conclusion} concludes this paper.

\section{Rate Constant Matrix Contraction Method}\label{sec:preliminaries}
This section provides preliminaries on the master equation~\eqref{def:master} and rate constant matrices.

\paragraph{Notations.}
Let $\R$, $\Rp$, and $\Rpp$ denote the set of reals, non-negative reals, and positive reals, respectively.
For a positive integer $n$, define $[n] \defeq \set{1, \dotsc, n}$.

For $I \subseteq [n]$, let $\R^I$ be the $\card{I}$-dimensional real vector space having coordinates indexed by $I$.
We call $\R^I$ a \emph{coordinate subspace} of $\R^{J}$ for $I \subseteq J \subseteq [n]$.
For $I, J \subseteq [n]$, let $\R^{I \times J}$ be the set of $\card{I} \times \card{J}$ matrices whose rows and columns are indexed by $I$ and $J$, respectively.
We simply denote $\R^n \defeq \R^{[n]}$, $\R^{I \times n} \defeq \R^{I \times [n]}$, $\R^{n \times J} \defeq \R^{[n] \times J}$, and $\R^{n \times n} \defeq \R^{[n] \times [n]}$.
For a vector $y \in \R^n$, let $y_I$ denote the subvector of $y$ indexed by $I \subseteq [n]$, and $A_{IJ}$ the submatrix of $A \in \R^{n \times n}$ indexed by rows $I$ and columns $J$.
For brevity, we rewrite $y_{\set{i}}$, $A_{\set{i}J}$, $A_{I\set{j}}$, and $A_{\set{i}\set{j}}$ as $y_i$, $A_{iJ}$, $A_{Ij}$, and $A_{ij}$, respectively, for $I, J \subseteq [n]$ and $i, j \in [n]$.
For $I \subseteq [n]$, we mean by $A_{II}^{-1}$ the inverse of $A_{II}$.
The transpose of a vector $a$ and a matrix $A$ are denoted by $\trsp{a}$ and $\trsp{A}$, respectively.

Let $\ones_n$ be the all-one vector in the $n$-dimensional space and $I_n$ the $n \times n$ identity matrix.
For $J \subseteq [n]$, let $\ones_J$ and $I_J$ be the all-one vector in $\R^J$ and the identity matrix in $\R^{J \times J}$, respectively.
Let $\zeros$ and $O$ be the zero vector and the zero matrix of an appropriate dimension (size), respectively.

Vectors and matrices are called \emph{non-negative} if their components (entries) are non-negative.
For $y \in \R^n$, we mean by $y \ge \zeros$ that $y$ is a non-negative vector.
Similarly, $A \ge O$ for $A \in \R^{n \times m}$ means that $A$ is a non-negative matrix.
For a row or column vector $y$, $\diag(y)$ denotes the diagonal matrix obtained by arranging the components of $y$ on the diagonals.

Let $\Delta_n$ denote the $(n-1)$-dimensional probability simplex and $\Delta_n^\circ
$ be its relative interior, i.e., $\Delta_n = \set[\big]{y \in \Rp^n}[\onestr_n y = 1]$ and $\Delta_n^\circ  = \set[\big]{y \in \Rpp^n}[\onestr_n y = 1]$.
\subsection{Rate Constant Matrices}\label{sec:rate-constant-matrices}

Let $K \in \R^{n \times n}$ be an $n \times n$ matrix satisfying the following conditions:
\begin{enumerate}[{label={(RCM\arabic*)}}]
  \item $K_{ij} \ge 0$ for $i \ne j$,\label{item:rcm1}
  \item $\onestr_n K = \zerostr$,\label{item:rcm2}
  \item the detailed balance condition~\eqref{def:detailed-balance} holds for some $\pi \in \Delta_n^\circ$.\label{item:rcm3}
\end{enumerate}
We call such $K$ and $\pi$ a \emph{rate constant matrix} and a \emph{stationary distribution}, respectively (see~\cite{Wei1969-kr}).
A matrix satisfying~\ref{item:rcm1} is called a \emph{Metzler matrix}~\cite{Farina2000-fk} or an \emph{essentially non-negative matrix}~\cite{Varga2000-yp}.
Setting
\begin{align}\label{def:diag-pi}
  \Pi \defeq \diag(\pi),
\end{align}
we can simply write~\ref{item:rcm3} as $K\Pi = \Pi\trsp{K}$, i.e., $K$ is a self-adjoint matrix with respect to the \emph{$\pi$-inner product} in the following sense.

For $a, b \in \R^n$, we define the \emph{$\pi$-inner product} by
\begin{align}\retainlabel{def:pi-inner-product}
  \pipr{a, b} \defeq \trsp{a} \Pi^{-1} b.
\end{align}
The $\pi$-inner product induces the \emph{$\pi$-norm} $\pinorm{a} \defeq \sqrt{\pipr{a, a}}$.
Unless otherwise stated, we regard $\R^n$ as a metric vector space equipped with the $\pi$-inner product.
Similar arguments to the standard inner product hold for basic notions such as adjoint matrices and positive definiteness in this space.
Specifically, the \emph{adjoint matrix} of $A \in \R^{n \times n}$ is the matrix $A^*$ such that $\pipr{a, Ab} = \pipr{A^*a, b}$ for all $a, b \in \R^n$, and is explicitly written as $A^* = \Pi \trsp{A} \Pi^{-1}$.
A matrix $A \in \R^{n \times n}$ is \emph{self-adjoint} if $A = A^*$, i.e., $A\Pi = \Pi A^\top$.
Thus,~\ref{item:rcm3} is nothing but a statement that $K$ is self-adjoint.
See \Cref{sec:adjointness,sec:positive-semi-definiteness} for details of adjoint matrices and positive semi-definiteness on the $\pi$-inner product space.

Taking the transpose on both sides of~\ref{item:rcm2}, we get $\zeros = \trsp{K} \ones_n = \Pi^{-1}K\Pi \ones_n = \Pi^{-1}K\pi$ and thereby $K \pi = \zeros$.
Conversely, a self-adjoint matrix $K \in \R^{n \times n}$ satisfying $K \pi = \zeros$ meets~\ref{item:rcm2}.
Thus, $\onestr_n K = \zerostr$ and $K \pi = \zeros$ are equivalent under~\ref{item:rcm3}.

The eigenvalues of a self-adjoint matrix are real.
A self-adjoint matrix $A \in \R^{n \times n}$ is called \emph{positive semi-definite} if $\pipr{a, Aa}$ is non-negative for $a \in \R^n$, or equivalently, all eigenvalues of $A$ are non-negative.
It follows from the Gershgorin circle theorem that a rate constant matrix $K$ has non-positive eigenvalues, i.e., $K$ is \emph{negative semi-definite}.
Let $0 = \lambda_1 \ge \dotsb \ge \lambda_n$ be the eigenvalues of $K$ and $u_k$ an eigenvector corresponding to $\lambda_k$ for $k \in [n]$.
Since $K$ is self-adjoint, we can take $u_1, \dotsc, u_n$ such that $\pipr{u_k, u_l}$ is $1$ if $k = l$ and $0$ otherwise.
We can set $u_1 = \pi$ by $K\pi = \zeros$.
Note that $\pi$ satisfies $\pinorm{\pi}^2 = \trsp{\pi} \Pi^{-1} \pi = \onestr_n \pi = 1$, i.e., $\pi$ is normalized.
See~\Cref{sec:positive-semi-definiteness} for details on the semi-definiteness of self-adjoint matrices.

The minus of $K$ is a \emph{Z-matrix}, i.e., every off-diagonal entry is non-negative.
In addition, $-K$ has non-negative eigenvalues.
Such a matrix is called an \emph{M-matrix}~\cite[Chapter~6]{Berma1994-xm}.
Principal submatrices of an M-matrix are M-matrices again, and the inverse of a nonsingular M-matrix is a non-negative matrix.
Therefore, $K_{SS}^{-1} \le O$ holds for $S \subseteq [n]$, provided that $K_{SS}$ is nonsingular.

The matrix $L \defeq -K\Pi$ is a (weighted) \emph{graph Laplacian matrix}, i.e., a symmetric Z-matrix such that $L\ones_n = \zeros$.
Conversely, given a graph Laplacian matrix $L \in \R^{n \times n}$ and $\pi \in \Delta_n^\circ$, we can construct a rate constant matrix $K$ by $K \defeq -L\Pi^{-1}$.

\subsection{Master Equations with Detailed Balance}\label{sec:master-equations-with-detailed-balance}
Consider the master equation~\eqref{def:master} associated with a rate constant matrix $K$ with stationary distribution $\pi$.
Letting $p \in \R^n$ be an initial value $x(0)$, the solution $x(t)$ of~\eqref{def:master} can be analytically represented by $x(t) = \e^{tK}p$, where $\e^{tK} \defeq \sum_{d=0}^\infty \frac{1}{d!} {(tK)}^{d}$.
The matrix exponential $\e^{tK}$ enjoys the following properties:

\begin{enumerate}[{label={(E\arabic*)}}]
  \item $\e^{tK} \ge O$ for $t \ge 0$,\label{item:e1}
  \item $\onestr_n \e^{tK} = \onestr_n$,\label{item:e2}
  \item $\e^{tK}$ is self-adjoint.\label{item:e3}
\end{enumerate}

We can directly derive~\ref{item:e2} and~\ref{item:e3} from~\ref{item:rcm2} and~\ref{item:rcm3}, respectively.
Property~\ref{item:e1} follows from the general fact that the matrix exponential of a Metzler matrix is non-negative~\cite[Section~8.2]{Varga2000-yp}.
Note that~\ref{item:e2} is equivalent to $\e^{tK} \pi = \pi$ under~\ref{item:e3}, which can be shown in the same way as the equivalence of~\ref{item:rcm2} and $K\pi = \zeros$ under~\ref{item:rcm3}.
These properties on $\e^{tK}$ give rise to the following well-known properties on the solution $x(t)$ of~\eqref{def:master}.
We here describe a short proof to clarify from which property of $\e^{tK}$ each claim is derived.

\begin{proposition}\label{prop:solution}
  Let $x(t)$ be the solution of~\eqref{def:master} with $x(0) = p \in \R^n$.
  If $p \in \Delta_n$, then $x(t) \in \Delta_n$ for all $t \ge 0$.
  If $p = \pi$, then $x(t) = \pi$ for all $t \ge 0$.
\end{proposition}

\begin{proof}
  By~\ref{item:e1}, we have $x(t) = \e^{tK}p \ge \zeros$ for $p \ge \zeros$ and $t \ge 0$.
  By~\ref{item:e2}, it holds that $\onestr_n x(t) = \onestr_n \e^{tK} p = \onestr_n p = 1$ if $\onestr_n p = 1$.
  By~\ref{item:e2} and~\ref{item:e3}, we have $\e^{tK} \pi = \pi$, which implies $x(t) = \e^{tK} \pi = \pi$ for $p = \pi$.
\end{proof}

Another important property of~\eqref{def:master} is the \emph{linearity} of solutions.
That is, if $x(t)$ and $y(t)$ are the solutions of~\eqref{def:master} with initial values $x(0) = p$ and $y(0) = q$, respectively, then $ax(t) + by(t)$ is the solution of~\eqref{def:master} with initial value $ap + bq$ for $a, b \in \R$.
We also remark that the eigenvalues of $\e^{tK}$ with $t \ge 0$ are in $[0, 1]$ because $K$ is negative semi-definite.

\subsection{Quasi-Steady-State Approximation}\label{sec:qssa}
QSSA, originated in the work of Bodenstein~\cite{Bodenstein1913-tc}, is a classical tool for reducing the number of variables and stiffness in a (generally nonlinear) master equation.
We here describe the method specialized to our linear setting~\eqref{def:master}.

Suppose that we have a bipartition $\set{S, T}$ of $[n]$ separating the components of $x(t)$ into \emph{fast variables} $x_S(t)$ and \emph{slow variables} $x_T(t)$.
With this bipartition, the master equation~\eqref{def:master} can be displayed as
\begin{align}\label{eq:master-block}
  \left\{\begin{aligned}
    \dot{x}_S(t) &= K_{SS}x_S(t) + K_{ST}x_T(t), \\
    \dot{x}_T(t) &= K_{TS}x_S(t) + K_{TT}x_T(t).
  \end{aligned}\right.
\end{align}
Now, $x_S(t)$ is approximately regarded as ``steady state,'' i.e., $\dot{x}_S(t) \approx \zeros$.
The QSSA replaces the left-hand side $\dot{x}_S(t)$ of the first equation in~\eqref{eq:master-block} with $\zeros$.
Namely, letting $\bar{x}(t)$ be a new variable vector, we can formulate the QSSA of~\eqref{eq:master-block} as
\begin{align}\label{eq:qssa}
  \left\{\begin{aligned}
                \zeros &= K_{SS}\bar{x}_S(t) + K_{ST}\bar{x}_T(t), \\
    \dot{\bar{x}}_T(t) &= K_{TS}\bar{x}_S(t) + K_{TT}\bar{x}_T(t).
  \end{aligned}\right.
\end{align}
The system~\eqref{eq:qssa} consists of purely algebraic equations and differential equations.
Such a system of equations is called a \emph{differential-algebraic equation} (DAE).
The algebraic equations constrain the solution to the \emph{QSSA subspace}
\begin{align}\label{def:slow-space}
  W \defeq \set[\big]{y \in \R^n}[K_{SS}y_S + K_{ST}y_T = \zeros].
\end{align}

If $K_{SS}$ is nonsingular, $\bar{x}_S(t)$ can be determined from $\bar{x}_T(t)$ by
\begin{align}\label{eq:x_S-from-x_T}
  \bar{x}_S(t) = -K_{SS}^{-1} K_{ST} \bar{x}_T(t).
\end{align}
Substituting~\eqref{eq:x_S-from-x_T} back into the differential equation in~\eqref{eq:qssa}, we obtain a differential equation involving only $\bar{x}_T(t)$:
\begin{align}\label{eq:qssa-eliminated}
  \dot{\bar{x}}_T(t) = \prn[\big]{K_{TT} - K_{TS}K_{SS}^{-1}K_{ST}} \bar{x}_T(t).
\end{align}
The coefficient matrix $D \defeq K_{TT} - K_{TS}K_{SS}^{-1}K_{ST}$ in~\eqref{eq:qssa-eliminated} is called the \emph{Schur complement} of $K_{SS}$ in $K$.
The Schur complement $D$ is again a rate constant matrix in the sense that it satisfies~\ref{item:rcm1}--\ref{item:rcm3} with stationary distribution proportional to $\pi_T$.
Hence, QSSA can be applied repeatedly.
From the property of the Schur complement, the equation obtained by applying the QSSA twice with $S = I$ and $J$ coincides with the single application of the QSSA with $S = I \cup J$ for any disjoint subsets $I, J \subseteq [n]$.

\section{RCMC Method}\label{sec:reformulating-matrix-update-formulas}

In this section, we describe the RCMC method in terms of matrix operations.
Readers interested in the original form of the method given by Sumiya et al.~\cite{Sumiya2015-br,Sumiya2017-qo} and its correspondence to our form are referred to \Cref{sec:original-rcmc}.

\subsection{Method Description}

\begin{algorithm}[tb]
	\caption{The RCMC method}\label{alg:improved}
	\begin{algorithmic}[1]
    \Input A rate constant matrix $K \in \R^{n \times n}$, an initial value $p \in \Delta_n$, and an end time $t_{\mathrm{max}} \in \Rp$.
    \Output $t^{(0)}, t^{(1)}, \dotsc \in \Rp$ and $q^{(0)}, q^{(1)}, \dotsc \in \R^n$.
    \State Output $t^{(0)} \defeq 0$ and $q^{(0)} \defeq p$
    \State $S^{(0)} \defeq \emptyset$, $T^{(0)} \defeq [n]$, $D^{(0)} \defeq K$
    \For{$k = 1, 2, \dotsc$}
      \State Take $j^{(k)} \in \argmax\set[\big]{\abs[\big]{D^{(k-1)}_{jj}}}[j \in T^{(k-1)}]$ {\label{line:improved-j}}
      \State $S^{(k)} \defeq S^{(k-1)} \cup \set[\big]{j^{(k)}}$ and $T^{(k)} \defeq T^{(k-1)} \setminus \set[\big]{j^{(k)}}$
      \State $\displaystyle D^{(k)} \defeq D^{(k-1)} - \frac{D^{(k-1)}_{T^{(k)}j^{(k)}}D^{(k-1)}_{j^{(k)}T^{(k)}}}{D^{(k-1)}_{j^{(k)}j^{(k)}}}$ {\label{line:improved-D}}
      \State Compute $t^{(k)}$ by~\eqref{def:t-k-diag} or~\eqref{def:t-k-eig} {\label{line:improved-t}}
      \If{$t^{(k)} > t_{\mathrm{max}}$} {\label{line:improved-if}}
        \State \textbf{break} {\label{line:improved-break}}
      \EndIf
      \State Output $t^{(k)}$ and $\displaystyle q^{(k)} \defeq \proj(Vp)$, where $V$ is the matrix~\eqref{def:V} of a designated type (Type~A or~B) {\label{line:improved-q}}
    \EndFor
	\end{algorithmic}
\end{algorithm}

\Cref{alg:improved} describes the RCMC method.
Given a rate constant matrix $K \in \R^{n \times n}$ with stationary distribution $\pi \in \Delta_n^\circ$, an initial value $p \in \R^n$, and the end time $t_{\mathrm{max}} \in \Rp$ as input, the method yields a sequence of non-negative reals $0 = t^{(0)} < t^{(1)} < t^{(2)} < \dotsb \le t_{\mathrm{max}}$ and vectors $q^{(0)}, q^{(1)}, \dotsc \in \R^n$ such that $q^{(k)}$ approximates the solution $x\prn[\big]{t^{(k)}}$ of the master equation~\eqref{def:master} with initial value $x(0) = p$ for each $k \ge 0$.

The RCMC method maintains a growing set $S^{(k)} \subseteq [n]$, its complement $T^{(k)} = [n] \setminus S^{(k)}$, and a matrix $D^{(k)} \in \R^{T^{(k)} \times T^{(k)}}$, where the superscripts $(k)$ are labeled to refer to the values at the $k$th iteration.
As with QSSA, $S^{(k)}$ and $T^{(k)}$ represents the sets of ``steady'' and ``transient'' states, respectively.
The matrix $D^{(k)}$ is the Schur complement of $K_{S^{(k)}S^{(k)}}$ in $K$, which is the coefficient matrix of the equation~\eqref{eq:qssa-eliminated} obtained by applying the QSSA with bipartition $\set[\big]{S^{(k)}, T^{(k)}}$.

In the $k$th iteration, the algorithm finds the index of a diagonal entry in $D^{(k-1)}$ with the largest absolute value and designates it as the next steady state $j^{(k)}$ (\cref{line:improved-j}).
This way of selecting steady states stems from a standard heuristic used in QSSA that considers $1 / \abs[\big]{D^{(k-1)}_{jj}}$ as the ``timescale'' of state $j \in T^{(k-1)}$ and selects the minimum-timescale state as the next steady state~\cite{Turanyi2014-ky}.
Additional justification of this method, based on error analysis and the connection to submodular function maximization, is discussed in \Cref{sec:map-inference-dpp}.

We should choose the reference time $t^{(k)}$ at \cref{line:improved-t} so that states in $S = S^{(k)}$ and $T = T^{(k)}$ can be considered steady and transient, respectively, at time $t^{(k)}$.
Following the time-scale heuristics in QSSA, one might set $t^{(k)}$ as
\begin{align}\label{def:t-k-diag}
  t^{(k)} \gets t_\mathrm{diag}^{(k)} \coloneqq \frac{1}{\abs[\big]{D^{(k-1)}_{j^{(k)}j^{(k)}}}}.
\end{align}
The original description of the RCMC method~\cite{Sumiya2015-br,Sumiya2017-qo} essentially adopts this choice of $t^{(k)}$.
In contrast, we will show in \Cref{sec:error-analysis} that the following choice of $t^{(k)}$ is better in terms of the error bound:
\begin{align}\label{def:t-k-eig}
  t^{(k)} \gets t_\mathrm{eigen}^{(k)}
  \defeq
  \frac{\ln 2}{\sqrt{\sigma\prn[\big]{K_{SS}} \rho\prn[\big]{D^{(k)}}}},
\end{align}
where $\rho(A)$ and $\sigma(A)$ for a square matrix $A$ denote the maximum and minimum absolute values of an eigenvalue of $A$, respectively.

The core of the RCMC method is the computation of the approximate solution $q^{(k)}$ at \cref{line:improved-q}.
Instead of the matrix exponential $\e^{tK}$ with $t = t^{(k)}$, we multiply $p$ by a matrix $V \in \R^{n \times n}$ that serves as an approximation to $\e^{tK}$.
We now introduce two variants of $V$, referred to as Types~A and~B, as follows:
\begin{gather}
  V = \begin{pmatrix}
    K_{SS}^{-1} K_{ST} V_{TT} K_{TS} K_{SS}^{-1} & -K_{SS}^{-1} K_{ST} V_{TT} \\
    -V_{TT} K_{TS} K_{SS}^{-1} & V_{TT}
  \end{pmatrix}\label{def:V}
\shortintertext{with}
  V_{TT} \defeq \begin{cases*}
    \diag \prn[\big]{\onestr_T M}^{-1} & (Type~A), \\
    M^{-1} & (Type~B),
  \end{cases*}\label{def:VTT}
\shortintertext{where}
  M \defeq I_T + K_{TS}K_{SS}^{-2}K_{ST} \in \R^{T \times T}.\label{def:M}
\end{gather}
Note that~\eqref{def:V} is displayed in such a way that blocks with rows (resp.\ columns) indexed by $S$ are placed above (resp.\ in the left of) those indexed by $T$.

\Cref{alg:improved} of Type~A indeed corresponds to the original RCMC method presented by Sumiya et al.~\cite{Sumiya2015-br,Sumiya2017-qo}; see \Cref{sec:original-rcmc} for this correspondence.
We see in \Cref{sec:apprixmating-matrix-exponential} that~\eqref{def:V} is obtained from natural requirements as an approximation of $\e^{tK}$ and the idea of QSSA.\@
The RCMC method of Type~A is practically advantageous in computational efficiency and numerical stability (see \Cref{sec:improved-rcmc-method}).
Meanwhile, the RCMC method of Type~B has several mathematically desirable properties in comparison with Type~A.
For example, the matrix $V$ of Type~B is the ($\pi$-norm) projection onto the QSSA subspace~\eqref{def:slow-space} with respect to $\set{S, T}$; see \Cref{sec:positive-semi-definiteness-of-V}.
In \Cref{sec:error-analysis}, we provide better and simpler error bounds for Type~B than Type~A.

The vector $Vp$ is guaranteed to be in $\Delta_n$ for Type~A (\cref{prop:Vp-in-Delta}), whereas it is not for Type~B.
Thus, after calculating $q \coloneqq Vp$, we compute its ($\pi$-norm) projection $\proj(q)$ onto $\Delta_n$, i.e.,
\begin{align}\label{def:proj}
  \proj(q) \defeq \argmin \set{\pinorm{q' - q}}[q' \in \Delta_n].
\end{align}
See \Cref{sec:projection} for the projection algorithm.
We will show in \Cref{prop:pythagorean} that the projection does not enlarge the distance from the exact solution $x\prn[\big]{t^{(k)}}$.

\subsection{Approximating Matrix Exponential}\label{sec:apprixmating-matrix-exponential}
We derive the matrix $V$ in the form of~\eqref{def:V} from several requirements that should be satisfied for $V$ to approximate $\e^{tK}$.
From~\ref{item:e1}--\ref{item:e3}, it is natural to expect the following properties to hold:
\begin{enumerate}[{label={(V\arabic*)}}]
  \item $V \ge O$,\label{item:v1}
  \item $\onestr_n V = \onestr_n$,\label{item:v2}
  \item $V$ is self-adjoint.\label{item:v3}
\end{enumerate}
Recall that~\ref{item:v2} is equivalent to $V\pi = \pi$ under~\ref{item:v3}.
Property~\ref{item:v1} is desirable to guarantee $Vp \ge \zeros$ for $p \ge \zeros$.
As seen in the proof of \Cref{prop:solution}, $\onestr_n V = \onestr_n$ and $V \pi = \pi$ correspond to the conservation law and the stationarity of $\pi$, respectively.
Furthermore, we require $V$ to satisfy the following:
\begin{enumerate}[{label={(V\arabic*)}, start=4}]
  \item $\Im V$ is contained in the QSSA subspace $W$ (defined by~\eqref{def:slow-space}) with respect to the bipartition $\set{S, T}$ of $[n]$.\label{item:v4}
\end{enumerate}
This property comes from the idea of QSSA; that is, if $V$ is a good approximation of $\e^{tK}$ and QSSA works well with $\set{S, T}$, then $Vp \approx \e^{tK}p = x(t)$ should be close to the QSSA subspace $W$.

We now design a matrix $V$ so that it satisfies these properties.
We first show that the form~\eqref{def:V} of $V$ is equivalent to~\ref{item:v3} and~\ref{item:v4}.

\begin{proposition}\label{prop:V-v1-v3}
  A matrix $V \in \R^{n \times n}$ satisfies~\ref{item:v3} and~\ref{item:v4} if and only if $V$ is in the form of~\eqref{def:V} with self-adjoint $V_{TT} \in \R^{T \times T}$.
\end{proposition}
\begin{proof}
  We first rephrase~\ref{item:v4}.
  The statement $Vp \in W$ is written down as
  \[
    \zeros
    = K_{SS} {(Vp)}_S + K_{ST} {(Vp)}_T
    = \prn[\big]{K_{SS} V_{SS} + K_{ST} V_{TS}} p_S + \prn[\big]{K_{SS} V_{ST} + K_{ST} V_{TT}} p_T.
  \]
  Since this equation must be satisfied for all $p \in \R^n$,~\ref{item:v4} is equivalent to
  \begin{align}
    K_{SS} V_{SS} + K_{ST} V_{TS} &= O,\label{eq:V-qssa-S}\\
    K_{SS} V_{ST} + K_{ST} V_{TT} &= O.\label{eq:V-qssa-T}
  \end{align}

  Suppose that $V$ satisfies~\ref{item:v3} and~\ref{item:v4}.
  Solving~\eqref{eq:V-qssa-T} for $V_{ST}$, we obtain $V_{ST} = -K_{SS}^{-1} K_{ST} V_{TT}$.
  Now we have $V_{TS} = -V_{TT} K_{TS} K_{SS}^{-1}$ due to~\ref{item:v3}.
  Substituting this into~\eqref{eq:V-qssa-S}, we have $V_{SS} = -K_{SS}^{-1} K_{ST} V_{TS} = K_{SS}^{-1} K_{ST} V_{TT} K_{TS} K_{SS}^{-1}.$
  Thus, $V$ must be in the form of~\eqref{def:V}.
  In addition, the self-adjointness of the principal submatrix $V_{TT}$ is necessary for~\ref{item:v3}.

  Conversely, suppose that a matrix $V$ is in the form of~\eqref{def:V} with self-adjoint $V_{TT}$.
  The self-adjointness~\ref{item:v3} of $V$ can be directly verified by using $K_{ST}^* = K_{TS}$, $K_{TS}^* = K_{ST}$ and the self-adjointness of $K_{SS}^{-1}$ and $V_{TT}$ to hold.
  The condition~\ref{item:v4} can be seen by checking~\eqref{eq:V-qssa-S} and~\eqref{eq:V-qssa-T}.
\end{proof}

By \Cref{prop:V-v1-v3}, the matrix $V$ satisfying~\ref{item:v3} and~\ref{item:v4} is uniquely determined from $V_{TT}$.
We next rephrase~\ref{item:v1} and~\ref{item:v2} to derive conditions on $V_{TT}$.
Recall that $M \in \R^{T \times T}$ is defined by~\eqref{def:M}.

\begin{proposition}\label{prop:V-v0-v2}
  For $V \in \R^{n \times n}$ in the form of~\eqref{def:V}, the following hold.
  \begin{enumerate}
    \item The matrix $V$ satisfies~\ref{item:v1} if and only if $V_{TT} \ge O$.\label{item:V-v0}
    \item The matrix $V$ satisfies~\ref{item:v2} if and only if $\onestr_T M V_{TT} = \onestr_T$.\label{item:V-v2}
  \end{enumerate}
\end{proposition}

\begin{proof}
  \ref{item:V-v0} The necessity is obvious.
  We show the sufficiency.
  Since $-K$ is an M-matrix, $-K_{SS}^{-1}$ is non-negative (see \Cref{sec:rate-constant-matrices}).
  Additionally, $K_{ST}$ and $K_{TS}$ are also non-negative by~\ref{item:rcm1}.
  Thus, $V \ge O$ follows from $V_{TT} \ge O$.

  \ref{item:V-v2} Using~\eqref{def:V}, we can express the equality $\onestr_n V = \onestr_n$ as
  \begin{align}
    \onestr_S K_{SS}^{-1} K_{ST} V_{TT} K_{TS} K_{SS}^{-1} - \onestr_T V_{TT} K_{TS} K_{SS}^{-1} &= \onestr_S,\label{eq:v2-rephrase-S}\\
    -\onestr_S K_{SS}^{-1} K_{ST} V_{TT} + \onestr_T V_{TT} &= \onestr_T.\label{eq:v2-rephrase-T}
  \end{align}
  If~\eqref{eq:v2-rephrase-T} is satisfied,~\eqref{eq:v2-rephrase-S} is simplified by using~\ref{item:rcm2} as
  \begin{align}\label{eq:ones_S-and-T}
    -\onestr_T K_{TS} K_{SS}^{-1} = \onestr_S.
  \end{align}
  Thus, $\onestr_n V = \onestr_n$ is equivalent to~\eqref{eq:v2-rephrase-T}.
  Substituting~\eqref{eq:ones_S-and-T} into~\eqref{eq:v2-rephrase-T}, we obtain $\onestr_T = \prn[\big]{\onestr_T - \onestr_S K_{SS}^{-1} K_{ST}} V_{TT} = \onestr_T M V_{TT}$.
\end{proof}

The matrices $V_{TT}$ of Types~A and~B given in~\eqref{def:VTT} both satisfy the second condition in \Cref{prop:V-v0-v2}.
Hence,~\ref{item:v2} holds on both types.
Regarding the first condition in \Cref{prop:V-v0-v2}, the matrix $V_{TT}$ of Type~A satisfies it due to the non-negativity of $M = I_T + K_{TS}K_{SS}^{-2}K_{ST}$, which leads us to the following proposition.

\begin{proposition}\label{prop:Vp-in-Delta}
  Let $V$ be the matrix~\eqref{def:V} of Type~A.
  Then, $Vp \in \Delta_n$ holds for any $p \in \Delta_n$.
\end{proposition}

Unfortunately, $M^{-1}$ is not necessarily non-negative, and hence~\ref{item:v1} is not satisfied for Type~B.
However, we can project $Vp$ onto $\Delta_n$ without making the distance to the exact solution larger.
\Cref{prop:Vp-in-Delta} helps us to reduce the computational cost and the complexity of the implementation for Type~A, which does not require the projection.

\subsection{Spectral Properties of \texorpdfstring{$V$}{V}}\label{sec:positive-semi-definiteness-of-V}
As described in \Cref{sec:master-equations-with-detailed-balance}, the matrix exponential $\e^{tK}$ with a rate constant matrix $K$ and $t \ge 0$ has eigenvalues in $[0, 1]$.
In this section, we see that the matrix $V$ of either type enjoys this spectral property.

We introduce additional notions on coordinate subspaces.
Since $\Pi^{-1}$ is diagonal, $\Pi_I^{-1}$ with $\Pi_I \defeq \diag\prn{\pi_I}$ is naturally regarded as a metric on $\R^I$ for any $I \subseteq [n]$.
The \emph{adjoint matrix} of $A \in \R^{I \times J}$ is then defined by $A^* \defeq \Pi_J \trsp{A} \Pi_I^{-1} \in \R^{J \times I}$.
Self-adjointness and positive semi-definiteness are similarly defined for matrices in $\R^{I \times I}$.
If $A \in \R^{n \times n}$ is self-adjoint, $A_{II}$ is also self-adjoint and $A_{IJ}^* = A_{JI}$ holds for any $I,J \subseteq [n]$.
If $A \in \R^{n \times n}$ is positive semi-definite, so is $A_{II}$.

For self-adjoint matrices $A, B \in \R^{I \times I}$, let $A \preceq B$ mean that $B - A$ is positive semi-definite.
As with the usual case of symmetric matrices, $\preceq$ forms a partial order on the self-adjoint matrices, called the \emph{Löwner order}, compatible with the addition.
See \Cref{sec:positive-semi-definiteness,sec:coordinate-subspaces} for details on the Löwner order and metrics of coordinate subspaces, respectively.

We first show that the two self-adjoint matrices $M$ and $V_{TT}$ are in the following relation of the Löwner order.

\begin{lemma}\label{lem:order-M-V_TT}
  The matrix $V_{TT}$ of either type satisfies $I_T \preceq M \preceq V_{TT}^{-1}$.
\end{lemma}

\begin{proof}
  Letting $X \defeq K_{SS}^{-1}K_{ST}$, we have $X^* = K_{TS} K_{SS}^{-1}$ and $M = I_T + X^* X$.
  Thus, $I_T \preceq M$ holds.
  We show $M \preceq V_{TT}^{-1}$.
  This is clear for Type~B since $M = V_{TT}^{-1}$.
  For Type~A, $M - V_{TT}^{-1} = M - \diag \prn[\big]{\onestr_T M}$ is a rate constant matrix with stationary distribution proportional to $\pi_T$, i.e., it satisfies~\ref{item:rcm1}--\ref{item:rcm3}.
  Therefore, $M - V_{TT}^{-1}$ is negative semi-definite.
\end{proof}

For a positive semi-definite self-adjoint matrix $A$, let $A^{\frac12}$ denote the \emph{square root} of $A$, i.e., the unique positive semi-definite matrix whose square is $A$.
Since $A \preceq B$ implies $P^* A P \preceq P^* B P$ for any $P \in \R^{n \times n}$, \Cref{lem:order-M-V_TT} can be expressed as follows.

\begin{lemma}\label{lem:order-M-V_TT-rephrase}
  The matrix $V_{TT}$ of either type satisfies $V_{TT} \preceq V_{TT}^{\frac12} M V_{TT}^{\frac12} \preceq I_T$.
\end{lemma}

The next proposition gives the same bound on the eigenvalues of $V$ as $\e^{tK}$.

\begin{proposition}\label{prop:v-positive-semi-definite}
  The matrix $V$ of either type satisfies $O \preceq V \preceq I_n$, i.e., all eigenvalues of $V$ are in $[0, 1]$.
\end{proposition}

To prove \cref{prop:v-positive-semi-definite}, we rewrite the matrix $V$ by using the following matrix
\begin{gather}
  \Omega \coloneqq \begin{pmatrix} -K_{TS} K_{SS}^{-1} & I_T\end{pmatrix} \in \R^{T \times n}\label{def:omega}
\shortintertext{with the adjoint}
  \Omega^* = \begin{pmatrix}
    \prn{-K_{ST} K_{SS}^{-1}}^* \\
    \prn{I_T}^*
  \end{pmatrix}
  = \begin{pmatrix}
    -\prn{K_{SS}^{-1}}^* \prn{K_{ST}}^* \\
    I_T
  \end{pmatrix}
  = \begin{pmatrix}
    -K_{SS}^{-1} K_{TS} \\
    I_T
  \end{pmatrix}.
\end{gather}
It can be directly seen that $V$ and $M$ are expressed as $V = \Omega^* V_{TT} \Omega$ and $M = \Omega \Omega^*$, respectively.

\begin{proof}[{Proof of \Cref{prop:v-positive-semi-definite}}]
  The positive semi-definiteness of $V$ follows from $V = \Omega^* V_{TT} \Omega$ and $V_{TT} \succeq O$.
  To show $V_{TT} \preceq I_n$, let $(\lambda, u)$ be an eigenpair of $V_{TT}^{\frac12} M V_{TT}^{\frac12}$; $\lambda \ge 1$ holds from \Cref{lem:order-M-V_TT-rephrase}.
  By direct calculation, we obtain $V\Omega^*V_{TT}^{\frac12}u = \lambda\Omega^*V_{TT}^{\frac12}u$, meaning that $\lambda$ is an eigenvalue of $V$ as well.
  Since $\Omega$ is of row-full rank and $V_{TT}$ is nonsingular, $\rank V = |T|$ holds, i.e., $V$ has $|T|$ non-zero eigenvalues and they are exactly the eigenvalues of $V_{TT}^{\frac12} M V_{TT}^{\frac12}$.
  Thus, $V \preceq I_n$ is obtained.
\end{proof}

As for Type~B, the eigenvalues of $V$ are either $0$ or $1$ because $V_{TT}^{\frac12} M V_{TT}^{\frac12} = I_T$ holds.
This means that $V$ is idempotent, i.e., $V^2 = V$.
Accordingly, the matrix $V$ of Type~B is the orthogonal (with respect to the $\pi$-inner product) projection onto $\Im V = \Im \Omega^*$, which is the QSSA subspace with respect to $\set{S, T}$.

\section{Error Analysis}\label{sec:error-analysis}
In this section, we provide upper bounds on the error between the output of the RCMC method and the analytic solution to the master equation~\eqref{def:master}.

\subsection{Results}\label{sec:errror-analysis-results}
Let $K \in \R^{n \times n}$ be a rate constant matrix with stationarity distribution $\pi \in \Delta_n^\circ$ and $V \in \R^{n \times n}$ the matrix~\eqref{def:V} of Type~A or~B with bipartition $\set{S, T}$ of $[n]$.
For an arbitrary $t \in \Rp$, we consider the following form of error:
\begin{align}\label{def:err}
  \Err_{S, t}(p) \coloneqq \frac{\pinormbig{\proj(Vp) - \e^{tK}p}}{\pinorm{p}} \quad (p \in \Delta_n),
\end{align}
where $\proj(\cdot)$ is defined by~\eqref{def:proj} and can be omitted for Type~A due to \Cref{prop:Vp-in-Delta}.
We provide bounds on the following two quantities:
\begin{itemize}
  \item $\Err_{S, t}(p)$ with a given $p \in \Delta_n$,
  \item the expectation $\E_p\sqbr{\Err_{S, t}(p)}$, where $p$ is drawn from the uniform distribution on the vertices of $\Delta_n$ (i.e.\ the standard vectors $e_1, \dotsc, e_n$).
\end{itemize}
The distribution assumed in the expectation error is natural in that a typical simulation of a kinetic scheme starts from a single state.

Our first step is to bound $\pinormbig{\proj(Vp) - \e^{tK}p}$ by $\pinormbig{\prn[\big]{V - \e^{tK}}p}$ as follows.

\begin{proposition}\label{prop:pythagorean}
  For $a \in \Delta_n$ and $q \in \R^n$, we have $\pinorm{\proj(q) - a} \le \pinorm{q - a}$.
\end{proposition}

\begin{proof}
  This inequality is a $\pi$-norm variant of the Pythagorean theorem, which follows from the standard Euclidean version as follows.
  Define $\Pi^{-\frac12}\Delta_n \defeq \set[\big]{\Pi^{-\frac12}y }[y \in \Delta_n]$, which one can easily prove to be convex.
  Note that we have $\proj(q) = \argmin\set[\big]{\norm[\big]{\Pi^{-\frac12}(y - q)}_\pi}[y \in \Delta_n] = \Pi^{\frac12} \argmin \set[\big]{\norm[\big]{y^\prime - \Pi^{-\frac12}q}_\pi}[y^\prime \in \Pi^{-\frac12}\Delta_n]$, i.e., $\Pi^{-\frac12} \proj(q)$ is the Euclidean projection of $\Pi^{-\frac12}q$ onto $\Pi^{-\frac12}\Delta_n$.
  Since $\Pi^{-\frac12}a \in \Pi^{-\frac12}\Delta_n$, the Pythagorean theorem implies
  $\norm[\big]{\Pi^{-\frac12}\proj(q) - \Pi^{-\frac12}a}_\pi \le \norm[\big]{\Pi^{-\frac12}q - \Pi^{-\frac12}a}_\pi$, which is nothing but the desired inequality.
\end{proof}
\Cref{prop:pythagorean} ensures that the projection $\proj(Vp)$ does not enlarge the distance between the approximate and exact solutions.
\Cref{prop:pythagorean} together with the spectral properties of $V$ and $\e^{tK}$ leads us to the following uniform error bound.
\begin{lemma}\label{lem:error-uniform-bound}
  For any $p \in \Delta_n$, it holds that $\Err_{S, t}(p) \le 1$.
\end{lemma}
We show \cref{lem:error-uniform-bound} in \Cref{sec:proof-lem-error-expand} with the aid of the notion of the \emph{matrix norm}.

We then focus on an error bound specific to a given $p \in \Delta_n$.
We first decompose $p$ in terms of the eigenvectors of $K$.

\begin{lemma}\label{lem:error-expand}
  Let $\set{u_1, \dotsc, u_n}$ be an orthonormal eigenbasis of $K$.
  The matrix $V$ of either type satisfies
  \begin{align}
    \frac{\pinormbig{\prn[\big]{V - \e^{tK}}p}}{\pinorm{p}}
    &\le
    \sum_{k=1}^n \frac{|\pipr{u_k, p}|}{\pinorm{p}} \pinormbig{\prn[\big]{V - \e^{tK}}u_k}
    \quad (p \in \Delta_n), \label{eq:error-expand}\\
    \E_p\sqbr{\frac{\pinormbig{\prn[\big]{V - \e^{tK}}p}}{\pinorm{p}}}
    &\le \sqrt{\frac{1}{n} \sum_{k=1}^n \pinormbig{\prn[\big]{V - \e^{tK}}u_k}^2}.\label{eq:exp-error-expand}
  \end{align}
\end{lemma}

For each eigenvector $u$ of $K$, the value $\pinormbig{\prn[\big]{V - \e^{tK}}u}$ is bounded as follows.

\begin{lemma}\label{lem:bound-on-V-minus-exp-u}
  For an eigenpair $(\lambda, u)$ of $K$ with $\pinorm{u} = 1$, we have
  \begin{align}
    \pinormbig{\prn[\big]{V - \e^{tK}}u} \le \min\set[\big]{
      1,
      \pinorm{Vu} + \e^{t\lambda},
      \pinorm{(I_n - V)u} + 1 - \e^{t\lambda}
    }.
  \end{align}
\end{lemma}

Proofs of \Cref{lem:error-expand,lem:bound-on-V-minus-exp-u} are also given in \Cref{sec:proof-lem-error-expand}.
The terms $\pinorm{Vu}$ and $\pinorm{(I_n - V)u}$ in~\eqref{eq:error-expand} and~\eqref{eq:exp-error-expand} are bounded by the following theorem.
Let $\norm{A}_{\infty, \mathrm{off}}$ denote the maximum absolute value of an off-diagonal entry of $A$.

\begin{lemma}\label{lem:bound-on-Vu}
  Let $D$ be the Schur complement of $K_{SS}$ in $K$.
  For an eigenpair $(\lambda, u)$ of $K$ with $\pinorm{u} = 1$, the following two statements hold.

  \begin{enumerate}
    \item Af for Type~A, we have
          \begin{align}
            \pinorm{Vu}       \le \frac{\rho(D)}{|\lambda|}, \quad
            \pinorm{(I_n - V)u} \le \frac{|\lambda| + \sqrt{2 \norm{\Pi^{-1} K}_{\infty, \mathrm{off}} |\lambda|}}{\sigma(K_{SS})}.
          \end{align}
    \item As for Type~B, we have
          \begin{align}
            \pinorm{Vu}       \le \frac{\rho(D)}{|\lambda|}, \quad
            \pinorm{(I_n - V)u} \le \frac{|\lambda|}{\sigma(K_{SS})}.
          \end{align}
  \end{enumerate}
\end{lemma}

\Cref{lem:bound-on-Vu} is proved in \Cref{sec:bounding-Vu,sec:bounding-I-V-u}.
It should be noted that the upper bounds on $\pinorm{Vu}$ and $\pinorm{(I_n-V)u}$ of either type given in \Cref{lem:bound-on-Vu} converge to $0$ as $\lambda \to -\infty$ and $\lambda \to 0$, respectively.
This fact agrees with the following observation: if $V$ approximates $\e^{tK}$ well, then
\begin{align}
  \pinorm{Vu} &\approx \pinormbig{\e^{tK} u} = \pinormbig{\e^{t\lambda} u} = \e^{t\lambda} \to 0 && (\lambda \to -\infty), \\
  \pinorm{(I_n-V)u} &\approx \pinormbig{(I_n - \e^{tK}) u} = \pinormbig{(1 - \e^{t\lambda}) u} = 1 - \e^{t\lambda} \to 0 && (\lambda \to 0).
\end{align}
Combining \Cref{prop:pythagorean,lem:error-expand,lem:bound-on-V-minus-exp-u,lem:bound-on-Vu,lem:error-uniform-bound}, we obtain the following.

\begin{theorem}\label{thm:error-main}
  Let $K \in \R^{n \times n}$ be a rate constant matrix with stationary distribution $\pi \in \Delta_n^\circ$, $\set{u_1, \dotsc, u_n}$ an orthonormal eigenbasis of $K$ with corresponding eigenvalues $\set{\lambda_1, \dotsc, \lambda_n}$, $\set{S, T}$ a bipartition of $[n]$, $D$ the Schur complement of $K_{SS}$ in $K$, and $V \in \R^{n \times n}$ the matrix~\eqref{def:V} of Type~A or~B.
  For $t \in \Rp$, it holds that
  \begin{align}
    \Err_{S, t}(p) &\le \min\set{\sum_{k=1}^n \frac{|\pipr{u_k, p}|}{\pinorm{p}} \min\set{1, \alpha_k, \beta_k}, 1}
    \quad (p \in \Delta_n), \label{eq:bound}\\
    \E_p\sqbr{\Err_{S, t}(p)} &\le \min\set{\sqrt{\frac{1}{n} \sum_{k=1}^n \min\set{1, \alpha_k, \beta_k}^2}, 1}
  \end{align}
  where
  \begin{align}
    \alpha_k \defeq \frac{\rho(D)}{|\lambda_k|} + \e^{t\lambda_k},
    \quad
    \beta_k \defeq \begin{cases}
      \frac{|\lambda_k| + \sqrt{2 \norm{\Pi^{-1} K}_{\infty, \mathrm{off}} |\lambda_k|}}{\sigma(K_{SS})} + 1 - \e^{t\lambda_k} & (\text{Type~A}), \\
      \frac{|\lambda_k|}{\sigma(K_{SS})} + 1 - \e^{t\lambda_k} & (\text{Type~B}) \label{def:alpha-beta}
    \end{cases}
  \end{align}
  for $k \in [n]$.
\end{theorem}

\Cref{thm:error-main} claims that the output of the RCMC method for an initial vector $p \in \Delta_n$ becomes a better approximation to the analytic solution $x(t) = \e^{tK}p$ if $p$ comprises few eigencomponents of $K$ with eigenvalues closer to $\sigma(K_{SS})$ or $\rho(D)$.
Our numerical experiments in \Cref{sec:experiments} show that~\eqref{eq:bound} for Type~B provides a sharp bound on the actual errors.

\Cref{thm:error-main} also leads to a principled way of determining the appropriate time $t^*$ corresponding to $\set{S, T}$ as follows.
Since
\begin{gather}
  \E_p\sqbr{\Err_{S, t}(p)}
  \le \sqrt{\frac{1}{n} \sum_{k=1}^n \min\set{1, \alpha_k, \beta_k}^2}
  \le \sqrt{\frac{1}{n} \sum_{k=1}^n f(t)^2}
  = f(t)
\shortintertext{with}
  f(t) \defeq \max_{\lambda < 0} \min \set{\frac{\rho(D)}{|\lambda|} + \e^{t\lambda}, \frac{|\lambda|}{\sigma(K_{SS})} + 1 - \e^{t\lambda}}\label{eq:minimize-error-problem}
\end{gather}
holds on Type~B, it is a natural idea to set $t$ as the minimizer $t^*$ of~\eqref{eq:minimize-error-problem}.
By calculation (see~\Cref{sec:determining-optimal-time}), we have
\begin{align}\label{def:t}
  t^* = \frac{\ln 2}{\sqrt{\sigma(K_{SS})\rho(D)}}.
\end{align}
The formula of $t^{(k)}_{\mathrm{eigen}}$ in~\eqref{def:t-k-eig} is derived in this way.
While it is possible to consider the optimal time based on the upper bound for Type~A, as we will see in \Cref{sec:experiments}, $\Err_{S, t}(p)$ on Type~A also falls below the upper bound for Type~B in practice.
Therefore, using $t^{(k)}_{\mathrm{eigen}}$ is practical even for Type~A.

\subsection{Proofs of \texorpdfstring{\Cref{lem:error-uniform-bound,lem:error-expand,lem:bound-on-V-minus-exp-u}}{Lemmas \ref{lem:error-uniform-bound} to \ref{lem:bound-on-V-minus-exp-u}}}\label{sec:proof-lem-error-expand}

This section gives proofs of \Cref{lem:error-uniform-bound,lem:error-expand,lem:bound-on-V-minus-exp-u}.
We denote $V - \e^{tK}$ by $E$ for notational simplicity.

Here, we introduce the notion of the \emph{matrix norm}.
For $I, J \subseteq [n]$, consider a matrix $A \in \R^{I \times J}$.
The \emph{matrix norm} of $A$ is defined by
\begin{align}\label{def:matrix-norm}
  \pinorm{A}
  \defeq \max_{v \in \R^J \setminus \set{\zeros}} \frac{\pinorm{Av}}{\pinorm{v}}
  = \sqrt{\rho(AA^*)}
  = \sqrt{\rho(A^*A)}.
\end{align}
Note that different norms are used in the numerator and the denominator in~\eqref{def:matrix-norm} unless $I = J$.
If $I = J$ and $A$ is self-adjoint, $\pinorm{A}$ is equal to $\rho(A)$.
See \Cref{sec:matrix-norm} for details.

\begin{proof}[{Proof of \Cref{lem:error-uniform-bound}}]
  By \cref{prop:pythagorean}, it suffices to show $\pinorm{E} \le 1$.
  Since $E$ is self-adjoint, we have $\pinorm{E} = \rho(E)$.
  Recall from \Cref{sec:master-equations-with-detailed-balance,prop:v-positive-semi-definite} that $O \preceq \e^{tK} \preceq I_n$ and $O \preceq V \preceq I_n$ hold.
  Hence, we have $E = V - \e^{tK} \preceq I_n - O = I_n$ and $E = V - \e^{tK} \succeq O - I_n = I_n$, meaning that $\rho(E) \le 1$ as required.
\end{proof}

\begin{proof}[{Proof of \Cref{lem:error-expand}}]
  Since $\set{u_1, \dotsc, u_n}$ is an orthonormal basis of $\R^n$, $p$ can be expanded as
  \begin{align}\label{eq:expand-p}
    p = \sum_{k=1}^n \pipr{u_k, p} u_k
  \end{align}
  (see \Cref{sec:adjointness}).
  Substituting~\eqref{eq:expand-p} into $\pinorm{Ep}$, we get
  \begin{align}
    \frac{\pinorm{Ep}}{\pinorm{p}}
    = \frac{1}{\pinorm{p}} \pinorm{\sum_{k=1}^n \pipr{u_k, p} Eu_k}
    \le \sum_{k=1}^n \frac{|\pipr{u_k, p}|}{\pinorm{p}} \pinorm{Eu_k}
  \end{align}
  by the triangle inequality.
  Thus,~\eqref{eq:error-expand} is proved.

  We next consider the expectation error.
  It follows from Jensen's inequality that
  \begin{align}
    \E_p\sqbr{\frac{\pinorm{Ep}}{\pinorm{p}}}^2 \le \E_p\sqbr{\frac{\pinorm{Ep}^2}{\pinorm{p}^2}}.
  \end{align}
  Since $\pinorm{Ep}^2$ for $p \in \Delta_n$ can be written as
  \begin{align}
    \pinorm{Ep}^2
    = \sum_{k=1}^n \sum_{l=1}^n \pipr{u_k, p} \pipr{u_l, p} \pipr{Eu_k, Eu_l}
  \end{align}
  by~\eqref{eq:expand-p}, we have
  \begin{align}\label{eq:exp-expand}
    \E_p\sqbr{\frac{\pinorm{Ep}^2}{\pinorm{p}^2}}
    = \sum_{k=1}^n \sum_{l=1}^n \E_p\sqbr{\frac{\pipr{u_k, p} \pipr{u_l, p}}{\pinorm{p}^2}} \pipr{Eu_k, Eu_l}.
  \end{align}
  Recall that the expectation is taken over the uniform distribution on the vertices $\set{e_1, \dotsc, e_n}$ of $\Delta_n$, where $e_i \in \R^n$ is the vector whose $i$th component is $1$ and others are $0$ for $i \in [n]$.
  The expectation in the right-hand side of~\eqref{eq:exp-expand} is calculated as
  \[
    \E_p\sqbr{\frac{\pipr{u_k, p} \pipr{u_l, p}}{\pinorm{p}^2}}
    = \frac{1}{n} \sum_{i=1}^n \frac{\pipr{u_k, e_i} \pipr{u_l, e_i}}{\pinorm{e_i}^2}
    = \frac{1}{n} \sum_{i=1}^n \frac{{(u_k)}_i{(u_l)}_i}{\pi_i}
    = \frac{\pipr{u_k, u_l}}{n},
  \]
  which is $\frac{1}{n}$ if $k = l$ and $0$ otherwise.
  Hence, we have
  \begin{align}
    \E_p\sqbr{\frac{\pinorm{Ep}^2}{\pinorm{p}^2}}
    = \frac{1}{n} \sum_{k=1}^n \pipr{Eu_k, Eu_k}
    = \frac{1}{n} \sum_{k=1}^n \norm{Eu_k}^2
  \end{align}
  and the desired bound~\eqref{eq:exp-error-expand} can be obtained by taking the square root.
\end{proof}

\begin{proof}[{Proof of \Cref{lem:bound-on-V-minus-exp-u}}]
  We give three upper bounds on $\pinorm{Eu_k}$ which proves~\eqref{eq:error-expand}.
  First, from the self-adjointness of $E = V - \e^{tK}$, we have $\pinorm{Eu_k} \le \pinorm{E} \pinorm{u_k} = \rho(E)$.
  By $O \preceq V, \e^{tK} \preceq I_n$, we get $-I_n \preceq E \preceq I_n$ and hence $\rho(E) \le 1$ holds.
  Next, we have $\pinorm{Eu_k} \le \pinorm{Vu_k} + \pinormbig{\e^{tK}u_k} = \pinorm{Vu_k} + \e^{t\lambda_k} \le \pinorm{Vu_k}$.
  Finally, we obtain $\pinorm{Eu_k} \le \pinorm{(I_n - V) u_k} + \pinormbig{\prn[\big]{I_n - \e^{tK}} u_k} = \pinorm{(I_n - V) u_k} + 1 - \e^{t\lambda_k}$.
\end{proof}

\subsection{Bounding \texorpdfstring{$\pinorm{Vu}$}{||Vu||\textunderscore π}}\label{sec:bounding-Vu}
We give a bound on $\pinorm{Vu}$ for an eigenvector $u$ of $K$.

\begin{lemma}
  For an eigenpair $(\lambda, u)$ of $K$ with $\pinorm{u} = 1$, on both Types~A and~B, it holds that $\pinorm{Vu} \le \dfrac{\rho(D)}{|\lambda|}$.
\end{lemma}

\begin{proof}
  Recall that the matrix $V$ can be expressed as $V = \Omega^* V_{TT} \Omega$, where $\Omega \in \R^{T \times n}$ is defined by~\eqref{def:omega}.
  Then, it holds that
  \begin{align}\label{eq:Vu-square}
    \pinorm{Vu}^2
    = \piprbig{u, V^2u}
    = \pipr{u, \Omega^*V_{TT}\Omega\Omega^*V_{TT}\Omega u}
    = \pipr{\Omega u, V_{TT}MV_{TT}\Omega u}.
  \end{align}
  By $K_{SS}u_S + K_{ST}u_T = \lambda u_S$ and $K_{TS}u_S + K_{TT}u_T = \lambda u_T$, we have $Du_T = \lambda\prn[\big]{u_T - K_{TS}K_{SS}^{-1}u_S} = \lambda \Omega u$.
  Substituting $\Omega u = \frac{1}{\lambda} Du_T$ into~\eqref{eq:Vu-square}, we get
  \begin{align}
    \pinorm{Vu}^2
    = \frac{1}{\lambda^2} \pipr{Du_T, V_{TT}MV_{TT}Du_T}
    = \frac{1}{\lambda^2} \pinormbig{M^{\frac12}V_{TT}Du_T}
  \end{align}
  and hence $\pinorm{v} \le \frac{1}{\abs{\lambda}} \pinormbig{M^{\frac12} V_{TT}^{\frac12}} \pinormbig{V_{TT}^{\frac12}} \pinorm{D} \pinorm{u_T}$ holds.
  By \Cref{lem:order-M-V_TT-rephrase}, we have $\pinormbig{M^{\frac12} V_{TT}^{\frac12}}^2 = \rho\prn[\big]{V_{TT}^{\frac12} M V_{TT}^{\frac12}} \le 1$ and $\pinormbig{V_{TT}^{\frac12}}^2 = \rho(V_{TT}) \le 1$.
  Since $\pinorm{D} = \rho(D)$ and $\pinorm{u_T} \le \pinorm{u} = 1$, we obtain the desired inequality.
\end{proof}

\subsection{Bounding \texorpdfstring{$\pinorm{(I_n - V)u}$}{||(I - V)u||\textunderscore π}}\label{sec:bounding-I-V-u}

We first consider Type~B.
As described in \Cref{sec:positive-semi-definiteness-of-V}, the matrix $V$ of Type~B is the orthogonal projection onto the QSSA subspace $\Im V$.
On the other hand, $I_n - V$ is the orthogonal projection onto $\ker V$.
Let $N \defeq I_S + K_{SS}^{-1}K_{ST}K_{TS}K_{SS}^{-1}$.
By the Woodbury matrix identity, it holds that
\begin{align}
  M^{-1} &= I_T - K_{TS}K_{SS}^{-1}N^{-1}K_{SS}^{-1}K_{ST}, \\
  N^{-1} &= I_S - K_{SS}^{-1}K_{ST}M^{-1}K_{TS}K_{SS}^{-1}, \\
  M^{-1}K_{TS}K_{SS}^{-1} &= K_{TS}K_{SS}^{-1}N^{-1}, \\
  K_{SS}^{-1}K_{ST}M^{-1} &= N^{-1}K_{SS}^{-1}K_{ST},
\end{align}
and hence $I_n - V$ can be expressed as
\begin{align}
  I_n - V
  &= \begin{pmatrix}
    I_S - K_{SS}^{-1}K_{ST}M^{-1}K_{TS}K_{SS}^{-1} & K_{SS}^{-1}K_{ST}M^{-1} \\
    M^{-1}K_{TS}K_{SS}^{-1} & I_T - M^{-1}
  \end{pmatrix} \\
  &= \begin{pmatrix}
    N^{-1} & N^{-1}K_{SS}^{-1}K_{TS} \\
    K_{TS}K_{SS}^{-1}N^{-1} & K_{TS}K_{SS}^{-1}N^{-1}K_{SS}^{-1}K_{ST}
  \end{pmatrix} = \Psi^* N^{-1} \Psi,\label{eq:In-V}
\end{align}
where $\Psi \defeq \begin{pmatrix} I_S & K_{SS}^{-1}K_{ST} \end{pmatrix}$.

\begin{lemma}\label{lem:bound-u-v-delta}
  For an eigenpair $(\lambda, u)$ of $K$ with $\pinorm{u} = 1$, the matrix $V$ of Type~B satisfies $\pinorm{(I_n - V)u} \le \dfrac{|\lambda|}{\sigma(K_{SS})}$.
\end{lemma}

\begin{proof}
  By~\eqref{eq:In-V}, we have
  \[
    \pinorm{(I_n - V)u}^2
    = \pinorm{\Psi^*N^{-1}\Psi u}^2
    = \pipr{\Psi u, N^{-1}\Psi\Psi^*N^{-1}\Psi u}
    = \pipr{\Psi u, N^{-1}\Psi u},
  \]
  where we used $\Psi\Psi^* = N$ in the last equality.
  The vector $\Psi u$ can be written as
  \[
    \Psi u
    = u_S + K_{SS}^{-1}K_{ST}u_T
    = K_{SS}^{-1}\prn[\big]{K_{SS}u_S + K_{ST}u_T}
    = \lambda K_{SS}^{-1} u_S.
  \]
  Thus, it holds $\pinorm{(I_n - V)u}^2 = \lambda^2 \pinormbig{N^{\frac12}K_{SS}^{-1}u_S}^2$ and
  \begin{align}
    \pinorm{(I_n - V)u}
    \le \abs{\lambda} \pinormbig{N^{\frac12}} \pinormbig{K_{SS}^{-1}} \pinorm{u_S}
    \le \frac{\abs{\lambda}}{\sigma(K_{SS})} \pinormbig{N^{\frac12}}.
  \end{align}
  Therefore, it suffices to show $\pinormbig{N^{-\frac12}} \le 1$.
  Since $N$ is self-adjoint, we have $\pinormbig{N^{-\frac12}}^2= \rho(N^{-1})$.
  In addition, $N \succeq I_S$ holds by definition, and thus $N^{-1} \preceq I_S$.
  Hence, we have $\rho(N^{-1}) \le 1$.
\end{proof}

We next consider Type~A.
For a vector $a \in \R^n$ and a positive vector $b \in \Rpp^n$, let $\frac{a}{b}$ denote the $n$-dimensional vector whose $i$th component is $\frac{a_i}{b_i}$ for $i \in [n]$.

\begin{lemma}
  Let $(\lambda, u)$ be an eigenpair of $K$ and $\Delta \defeq K \diag\prn{\frac{u}{\pi}} - \diag\prn{\frac{u}{\pi}}K \in \R^{n \times n}$.
  For the matrix $V$ of Type~A, we have $\pinorm{(I_n - V)u} \le \dfrac{\abs{\lambda} + \pinorm{\Delta}}{\sigma(K_{SS})}$.
\end{lemma}

\begin{proof}
  Let $v \defeq Vu$.
  First, we have
  \begin{align}
    u - v
    &= \begin{pmatrix} u_S + K_{SS}^{-1}K_{ST}v_T \\ u_T - v_T \end{pmatrix}
    = \begin{pmatrix} -K_{SS}^{-1}K_{ST}(u_T - v_T) + u_S + K_{SS}^{-1}K_{ST}u_T \\ u_T - v_T \end{pmatrix} \\
    &= \begin{pmatrix} -K_{SS}^{-1}K_{ST}(u_T - v_T) + \lambda K_{SS}^{-1} u_S \\ u_T - v_T \end{pmatrix}
    = \Omega^*(u_T - v_T) + \lambda\begin{pmatrix} K_{SS}^{-1} u_S \\ O \end{pmatrix}.\label{eq:delta-bound-0}
  \end{align}
  Since $\pinorm{\Omega^*(u_T - v_T)} = \pinormbig{M^{\frac12}(u_T - v_T)}$ by $\Omega\Omega^* = M$, we obtain
  \begin{align}\label{eq:delta-bound-u-v}
    \pinorm{u - v}
    \le \pinormbig{M^{\frac{1}{2}}\prn{u_T - v_T}} + \frac{\abs{\lambda}}{\sigma\prn{K_{SS}}}.
  \end{align}

  We next bound $\pinormbig{M^{\frac{1}{2}}\prn{u_T - v_T}}$.
  By $O \preceq M^{\frac{1}{2}}V_{TT}M^{\frac{1}{2}} \preceq I_T$ , we have
  \begin{align}\label{eq:delta-bound-M-to-V}
    \pinormbig{M^{\frac{1}{2}}\prn{u_T - v_T}}
    \le \pinormbig{M^{\frac{1}{2}}V_{TT}^{\frac{1}{2}}}\pinormbig{V_{TT}^{-\frac{1}{2}}\prn{u_T - v_T}}
    \le \pinormbig{V_{TT}^{-\frac{1}{2}}\prn{u_T - v_T}}.
  \end{align}
  By $\onestr_S K_{SS} + \onestr_T K_{TS} = 0$, we have
  \begin{gather}
    V_{TT}^{-1}
    = \diag\prn[\big]{\onestr_T M}
    = I_T + \diag\prn[\big]{\onestr_T K_{TS}K_{SS}^{-2}K_{ST}}
    = I_T - \diag\prn[\big]{\onestr_S K_{SS}^{-1}K_{ST}}
  \shortintertext{and}
    \begin{aligned}\label{eq:delta-bound-1}
      V_{TT}^{-\frac{1}{2}}(u_T - v_T)
      &= V_{TT}^{-\frac{1}{2}}\prn[\big]{u_T - V_{TT}\prn[\big]{u_T - K_{TS}K_{SS}^{-1}u_S}} \\
      &= V_{TT}^{\frac{1}{2}}\prn[\big]{V_{TT}^{-1}u_T - u_T + K_{TS}K_{SS}^{-1}u_S} \\
      &= V_{TT}^{\frac{1}{2}}\prn[\big]{K_{TS}K_{SS}^{-1}u_S - \diag\prn[\big]{\onestr_S K_{SS}^{-1}K_{ST}}u_T}.
    \end{aligned}
  \end{gather}
  By $\trsp{\prn{K_{ST}}} = \Pi_T^{-1}K_{TS}\Pi_S$ and $\trsp{\prn[\big]{K_{SS}^{-1}}} = \Pi_S^{-1}K_{SS}^{-1}\Pi_S$, we get
  \begin{align}\label{eq:delta-bound-2}
    \begin{aligned}
      \diag\prn[\big]{\onestr_S K_{SS}^{-1}K_{ST}}
      &= \diag\prn[\big]{\trsp{\prn[\big]{\onestr_S K_{SS}^{-1}K_{ST}}}} \\
      &= \diag\prn[\big]{\Pi_T^{-1}K_{TS}K_{SS}^{-1}\Pi_S\ones_S}
      = \diag\prn[\big]{K_{TS}K_{SS}^{-1}\pi_S}\Pi_T^{-1}.
    \end{aligned}
  \end{align}
  Combining~\eqref{eq:delta-bound-1} and~\eqref{eq:delta-bound-2}, we have
  \begin{align}
    V_{TT}^{-\frac{1}{2}}(u_T - v_T)
    &= V_{TT}^{\frac{1}{2}}\prn[\big]{K_{TS}K_{SS}^{-1}u_S - \diag\prn[\big]{K_{TS}K_{SS}^{-1}\pi_S}\tfrac{u_T}{\pi_T}} \\
    &= V_{TT}^{\frac{1}{2}}\prn[\big]{K_{TS}K_{SS}^{-1}u_S - \diag\prn[\big]{\tfrac{u_T}{\pi_T}}K_{TS}K_{SS}^{-1}\pi_S} \\
    &= V_{TT}^{\frac{1}{2}}\prn[\big]{K_{TS}K_{SS}^{-1}\diag\prn[\big]{\tfrac{u_S}{\pi_S}} - \diag\prn[\big]{\tfrac{u_T}{\pi_T}}K_{TS}K_{SS}^{-1}}\pi_S \\
    &= V_{TT}^{\frac{1}{2}}\prn[\big]{K_{TS}K_{SS}^{-1}\diag\prn[\big]{\tfrac{u_S}{\pi_S}} - K_{TS}\diag\prn[\big]{\tfrac{u_S}{\pi_S}}K_{SS}^{-1} \\
    &\qquad\qquad+ K_{TS}\diag\prn[\big]{\tfrac{u_S}{\pi_S}}K_{SS}^{-1} - \diag\prn[\big]{\tfrac{u_T}{\pi_T}}K_{TS}K_{SS}^{-1}}\pi_S \\
    &= V_{TT}^{\frac{1}{2}}\prn[\big]{K_{TS}K_{SS}^{-1}\prn[\big]{\diag\prn[\big]{\tfrac{u_S}{\pi_S}}K_{SS} - K_{SS}\diag\prn[\big]{\tfrac{u_S}{\pi_S}}} \\
    &\qquad\qquad+ K_{TS}\diag\prn[\big]{\tfrac{u_S}{\pi_S}} - \diag\prn[\big]{\tfrac{u_T}{\pi_T}}K_{TS}}K_{SS}^{-1}\pi_S \\
    &= V_{TT}^{\frac{1}{2}}\prn[\big]{-K_{TS}K_{SS}^{-1}\Delta_{SS} + \Delta_{TS}}K_{SS}^{-1}\pi_S \\
    &= V_{TT}^{\frac{1}{2}}\Omega\Delta\begin{pmatrix} K_{SS}^{-1}\pi_S \\ \zeros \end{pmatrix},
  \end{align}
  whose norm is bounded by
  \begin{align}\label{eq:delta-bound-V}
    \pinormbig{V_{TT}^{-\frac{1}{2}}(u_T - v_T)}
    \le \pinormbig{V_{TT}^{\frac{1}{2}}\Omega} \pinorm{\Delta} \pinorm{\begin{pmatrix} K_{SS}^{-1}\pi_S \\ \zeros \end{pmatrix}}.
  \end{align}
  By $V_{TT} \preceq M^{-1}$, it holds that
  \begin{gather}
    \pinormbig{V_{TT}^{\frac{1}{2}}\Omega}^2
    = \rho\prn[\big]{V_{TT}^{\frac{1}{2}}\Omega\Omega^*V_{TT}^{\frac{1}{2}}}
    = \rho\prn[\big]{V_{TT}^{\frac{1}{2}}MV_{TT}^{\frac{1}{2}}}
    \le 1\label{eq:delta-bound-a}\sbox0{\ref{eq:delta-bound-a}}
  \shortintertext{and}
    \pinorm{\begin{pmatrix} K_{SS}^{-1}\pi_S \\ \zeros \end{pmatrix}}
    = \pinormbig{K_{SS}^{-1}\pi_S}
    \le \pinormbig{K_{SS}^{-1}} \pinorm{\pi_S}
    \le \frac{1}{\sigma(K_{SS})}.\label{eq:delta-bound-b}
  \end{gather}
  By~\eqref{eq:delta-bound-u-v},~\eqref{eq:delta-bound-M-to-V}, and~\eqref{eq:delta-bound-V}--\eqref{eq:delta-bound-b}, we obtain the desired inequality.
\end{proof}

To bound $\pinorm{\Delta}$, we give two lemmas.
The first one is a generalization of the inequality that bounds the spectral norm of a matrix by the Frobenius norm.

\begin{lemma}\label{lem:spectral-frobenius}
  For $A \in \R^{n \times n}$, it holds that $\pinorm{A} \le \sqrt{\tr A^*A}$.
\end{lemma}

\begin{proof}
  The claim follows from the fact that $\pinorm{A}^2 = \rho(A^*A)$ is the maximum eigenvalue of $A^*A$, whereas $\tr A^*A$ is the sum of eigenvalues of $A^* A$.
\end{proof}

Recall that $L = -K\Pi$ is a graph Laplacian matrix.
It is well-known that the quadratic form $\trsp{a} L a$ of a graph Laplacian matrix with $a \in \R^n$ satisfies $\trsp{a}La = \frac12 \sum_{i \ne j} L_{ij} {(a_i - a_j)}^2$.
The following lemma is a variant of this identity for $K$.

\begin{lemma}\label{lem:laplacian-quadratic}
  For $a \in \R^n$, it holds that $\displaystyle \pipr{a, Ka} = -\frac{1}{2}\sum_{i \ne j} L_{ij}\prn{\frac{a_i}{\pi_i} - \frac{a_j}{\pi_j}}^2$.
\end{lemma}

\begin{proof}
  Using $L = -K\Pi$, we obtain
  \begin{align}
    \pipr{a, Ka}
      = -\trsp{a}\Pi^{-1}L\Pi^{-1}a
      = -\trsp{\prn{\frac{a}{\pi}}} L \frac{a}{\pi}
      = -\frac{1}{2} \sum_{i \ne j} L_{ij} \prn{\frac{a_i}{\pi_i} - \frac{a_j}{\pi_j}}^2.
  \end{align}
\end{proof}

We now give an upper bound on $\pinorm{\Delta}$, which, together with \Cref{lem:bound-u-v-delta}, leads to the upper bound on $\pinorm{(I_n - V)u}$ for Type~A in \Cref{lem:bound-on-Vu}.

\begin{lemma}
  Let $(\lambda, u)$ be an eigenpair of $K$ and $\Delta \defeq K \diag\prn{\frac{u}{\pi}} - \diag\prn{\frac{u}{\pi}}K \in \R^{n \times n}$.
  Then, it holds that $\pinorm{\Delta} \le \sqrt{2 \abs{\lambda} \norm[\big]{\Pi^{-1}K}_{\infty,\mathrm{off}}}$.
\end{lemma}

\begin{proof}
  By \Cref{lem:spectral-frobenius} and $-\Delta^* = \Delta$, it holds $\pinorm{\Delta}^2 \le \tr \Delta^*\Delta = -\tr \Delta^2$.
  The $(i, j)$ entry in $\Delta$ is
  \begin{align}
    \Delta_{ij}
  = K_{ij}\prn{\frac{u_j}{\pi_j} - \frac{u_i}{\pi_i}}
  = \frac{L_{ij}}{\pi_j}\prn{\frac{u_i}{\pi_i} - \frac{u_j}{\pi_j}}.
  \end{align}
  We thus have
  \begin{align}
    \pinorm{\Delta}^2
    &= -\sum_{i,j} \Delta_{ij} \Delta_{ji}
    = \sum_{i \ne j} \frac{L_{ij}^2}{\pi_i \pi_j} \prn{\frac{u_i}{\pi_i} - \frac{u_j}{\pi_j}}^2 \\
    &\le \max_{i \ne j} \frac{L_{ij}}{\pi_i \pi_j} \sum_{i \ne j} L_{ij} \prn{\frac{u_i}{\pi_i} - \frac{u_j}{\pi_j}}^2
    = -2\norm[\big]{\Pi^{-1}K}_{\infty,\mathrm{off}}\pipr{u, Ku},
  \end{align}
  where \Cref{lem:laplacian-quadratic} is used in the last equality.
  Since $\pipr{u, Ku} = \lambda$, we obtain the desired inequality.
\end{proof}

\section{Implementation Details of RCMC Method}\label{sec:improved-rcmc-method}
In this section, we discuss the implementation details of the RCMC method for fast and stable computation.

\subsection{Incremental Update of LU Decomposition}\label{sec:incremantally-updating-lu-decomposition}

In the calculation of the matrix-vector multiplication $Vp$, one needs to solve linear systems with the coefficient matrix $K_{SS} = K_{S^{(k)}S^{(k)}}$, which costs $\Order(k^3)$ operations if an LU decomposition of $K_{SS}$ is computed from scratch.
Fortunately, we can incrementally update an LU decomposition of $K_{SS}$ for successive $k$ in the following way.

We first discuss the LU decomposition of a self-adjoint positive semi-definite matrix $A \in \R^{n \times n}$ in terms of the $\pi$-norm metric.
The matrix $B \coloneqq A\Pi^{-1}$ is a symmetric positive semi-definite matrix.
Consider a Cholesky decomposition $P C C^\top P^\top$ of $B$, where $C \in \R^{n \times n}$ is a lower-triangular matrix and $P \in \R^{n \times n}$ is a permutation matrix representing the order of pivoting.
Then, $C\prn[\big]{C^\top\Pi^{-1}}$ is an LU decomposition of $P^\top AP$.
We regard the row and column spaces of $C$ as the linear spaces equipped with the metrics $P^\top \Pi^{-1} P = {\diag(P\pi)}^{-1}$ and $I_n$, respectively.
Then, the adjoint matrix of $C$ is given by $C^* = C^\top \Pi^{-1}$ and $CC^* = P^\top AP$ holds.
We call this decomposition a \emph{Cholesky decomposition} of $A$ with respect to the $\pi$-norm.
The matrix $C$ is called a \emph{Cholesky factor} of $A$.

Throughout the algorithm, we maintain an $n \times k$ matrix $C^{(k)} \in \R^{n \times S^{(k)}}$.
Initially, $C^{(0)}$ is set as the $n \times 0$ matrix.
At the $k$th iteration, we update $C^{(k-1)}$ to $C^{(k)}$ by appending a column vector as
\begin{align}\label{eq:update-cholesky}
  C^{(k)} \defeq \begin{pNiceArray}{c|c}[first-row,last-col]
    S^{(k-1)} & j^{(k)} \\
    \Block{2-1}{C^{(k-1)}} & \zeros & S^{(k-1)} \\
                           & -\sqrt{\frac{\pi_{j^{(k)}}}{\abs[\big]{D^{(k-1)}_{j^{(k)}j^{(k)}}}}}D^{(k-1)}_{T^{(k-1)}j^{(k)}} & T^{(k-1)}
  \end{pNiceArray}.
\end{align}

\begin{proposition}\label{prop:incremental-update-cholesky}
  The matrix $C^{(k)}$ given by~\eqref{eq:update-cholesky} is equal to $C_{[n]S^{(k)}}$ for some Cholesky factor $C$ of $-K$.
\end{proposition}
\begin{proof}
  Let $Q^{(k)} \defeq -D^{(k)}\Pi_{T^{(k)}} = L_{T^{(k)}T^{(k)}} - L_{T^{(k)}S^{(k)}}L^{-1}_{S^{(k)}S^{(k)}}L_{S^{(k)}T^{(k)}}$ be the Schur complement of $L_{S^{(k)}S^{(k)}}$ in $L = -K\Pi$.
  Then, we can write~\eqref{eq:update-cholesky} as
  \begin{align}\label{eq:update-cholesky-L}
    C^{(k)} = \begin{pNiceArray}{c|c}
      \Block{2-1}{C^{(k-1)}} & \zeros \\
                            & \frac{1}{\sqrt{Q^{(k-1)}_{j^{(k)}j^{(k)}}}} Q^{(k-1)}_{T^{(k-1)}j^{(k)}}
    \end{pNiceArray}.
  \end{align}
  The formula~\eqref{eq:update-cholesky-L} is identical to the procedure of computing a Cholesky decomposition $CC^\top$ of $L$ with pivoting order $j^{(1)}, j^{(2)}, \dotsc$; see, e.g.,~\cite[Chapter~10]{Higham2002-nv}.
  Thus, it holds that $C^{(k)} = C_{[n]S^{(k)}}$.
\end{proof}

By \Cref{prop:incremental-update-cholesky}, $C^{(k)}_{S^{(k)}S^{(k)}}$ provides a Cholesky factor of $-K_{S^{(k)}S^{(k)}}$.

The RCMC method of Type~B additionally involves solving a linear system with the coefficient matrix $M$ defined by~\eqref{def:M}.
In \Cref{sec:incremental-update-M}, we describe an efficient algorithm for incrementally updating a Cholesky factor of $M$.

\subsection{Numerical Stabilization}\label{sec:numerical-stabilization}
A rate constant matrix $K$ arising from chemical kinetics contains a wide range of numbers in its entries, which may lead to catastrophic cancellation when two like-sign numbers are subtracted in finite precision arithmetic~\cite{Higham2002-nv}.
Therefore, designing an algorithm that avoids subtracting like-sign numbers is crucial to numerical stability.

Fortunately, most matrices appearing in \Cref{alg:improved} of Type~A have constant sign patterns.
For instance, $K_{ST}$, $K_{TS}$, and $V_{TT} = \diag\prn[\big]{\onestr_T M}^{-1}$ are non-negative matrices for any $S = S^{(k)}$ and $T = T^{(k)}$.
The Cholesky factor $C$ of $-K$ treated in \Cref{sec:incremantally-updating-lu-decomposition} is a lower-triangular matrix whose diagonals and off-diagonals are non-negative and non-positive, respectively, by~\eqref{eq:update-cholesky} and the fact that $D = D^{(k)}$ is a rate constant matrix.
Thus, we have $C_{SS}^{-1} \ge O$ from the inversion formula of triangular matrices.
As a result, for any $p \ge \zeros$, the matrix-vector multiplication $Vp$ on Type~A is free from subtractions of like-signed numbers.

Unfortunately, the same property does not apply to the matrix-vector multiplication in Type~B due to the absence of the non-negativity of $M^{-1}$.
Nevertheless, in our experiments detailed in \Cref{sec:experiments}, we have not encountered any numerical instability.
Theoretically analyzing the numerical stability for the TypeB RCMC method is left as future work.

Apart from the matrix-vector multiplication, \Cref{alg:improved} performs arithmetic operations during the update of $D$ at \Cref{line:improved-D}.
Unlike the off-diagonal updates, the following diagonal update in \Cref{line:improved-D} is the subtraction of two non-positive numbers:
\begin{align}\label{eq:unstable-diagonal}
  D_{ii} \defeq D^{(k-1)}_{ii} - \frac{D_{ij}^{(k-1)}D_{ji}^{(k-1)}}{D_{jj}^{(k-1)}}
  \qquad \prn[\big]{i \in T},
\end{align}
where $j = j^{(k)}$.
Indeed, due to $\onestr_{T} D = \zerostr$, we can compute $D_{ii}$ from off-diagonals of $D$ as
\begin{align}\label{eq:stable-diagonal}
  D_{ii} = -\sum_{k \in T \setminus \set{i}} D_{ki}
  \qquad \prn[\big]{i \in T}.
\end{align}
Diagonal entries in $D$ should be updated via~\eqref{eq:stable-diagonal} in place of~\eqref{eq:unstable-diagonal} from the viewpoint of numerical stability.

\subsection{Faster Reference Time Computation}\label{sec:time-computation}
Computing $t_{\mathrm{eigen}}^{(k)}$ given by~\eqref{def:t-k-eig} costs longer time than other operations in \cref{alg:improved}, since it involves the computation of the spectral radii of $D$ and $K_{SS}^{-1}$.
When the running time is a significant concern, for a square matrix $A$, one can replace $\rho(A)$ and $\sigma(A)$ with
\begin{align}\label{def:rho-hat}
  \hat{\rho}(A) \coloneqq \min\set{
    \max_{i} \set{\sum_{j} \abs[\big]{A_{ij}}},
    \max_{j} \set{\sum_{i} \abs[\big]{A_{ij}}}
  }
  \quad \text{and} \quad
  \hat{\sigma}(A) \coloneqq \frac{1}{
    \hat{\rho}(A^{-1})
  },
\end{align}
respectively.
The value $\hat{\rho}(A)$ is an upper bound on $\rho(A)$ derived by the Gershgorin theorem.
Since $K_{SS}$ (with $S = S^{(k)}$) is an M-matrix, we further have
\begin{align}
  \hat{\rho}\prn[\big]{K_{SS}^{-1}} = \min\set{\max_{i \in S} \prn[\big]{K_{SS}^{-1}\ones_S}_i, \max_{j \in S} \prn[\big]{\onestr_S K_{SS}^{-1}}_j},
\end{align}
which can be computed by solving two linear systems with the coefficient matrix $K_{SS}$.

\section{Experiments}\label{sec:experiments}

In this section, we report numerical experiments of the RCMC method using synthetic and real-world datasets.
All experiments were conducted using \CC\ and \texttt{Eigen}\footnote{\texttt{\url{https://eigen.tuxfamily.org/}}} 3.4.0 compiled by \texttt{Apple Clang} 14.0.3 on a laptop powered by an Apple M2 CPU (8 Cores) with \SI{24}{GB} of RAM.

\subsection{Datasets and Settings}
\setlength{\arraycolsep}{2pt}

We employ three datasets: one small-scale synthetic and two large-scale real-world data.
The synthetic data consists of the following rate constant matrix $K \in \R^{6 \times 6}$ (expressed with two significant digits):
\[
  \prn{\begin{array}{rrrrrr}
    -2.0 \times 10^{9\phantom{-}} & 6.9 \times 10^{9\phantom{1}} & 0\phantom{{}\times 10^{11}} & 1.7 \times 10^{-18}         & 0\phantom{{}\times 10^{-9}}    & 0\phantom{{}\times 10^{-21}} \\
    2.0 \times 10^{9\phantom{-}}  & -3.9 \times 10^{11}          & 3.3 \times 10^{11}         & 0\phantom{{}\times 10^{-18}} & 1.9 \times 10^{7\phantom{-}}  & 0\phantom{{}\times 10^{-21}} \\
    0\phantom{{}\times 10^{-7}}    & 3.9 \times 10^{11}           & -3.3 \times 10^{11}        & 0\phantom{{}\times 10^{-18}} & 0\phantom{{}\times 10^{-9}}    & 0\phantom{{}\times 10^{-21}} \\
    9.6 \times 10^{-7}            & 0\phantom{{}\times 10^{11}}   & 0\phantom{{}\times 10^{11}} & -1.7 \times 10^{-18}        & 0\phantom{{}\times 10^{-9}}    & 0\phantom{{}\times 10^{-21}} \\
    0\phantom{{}\times 10^{-7}}    & 3.4 \times 10^{7\phantom{1}} & 0\phantom{{}\times 10^{11}} & 0\phantom{{}\times 10^{-18}} & -1.9 \times 10^{7\phantom{-}} & 8.1 \times 10^{-21} \\
    0\phantom{{}\times 10^{-7}}    & 0\phantom{{}\times 10^{11}}   & 0\phantom{{}\times 10^{11}} & 0\phantom{{}\times 10^{-18}} & 8.3 \times 10^{-9}            & -8.1 \times 10^{-21} \\
  \end{array}}
\]
with $\pi = {(9.1 \times 10^{-13},\, 2.6 \times 10^{-13},\, 3.1 \times 10^{-13},\, 0.51,\, 4.8 \times 10^{-13},\, 0.49)}^\top \in \Delta_6^\circ$.

\begin{table}
  \centering
  \caption{Properties of rate constant matrices $K \in \R^{n \times n}$, where $\nnz(K)$ is the number of non-zero entries in $K$.}\label{tab:datasets}
  \begin{tabular}{lrrrr}\toprule
    \multicolumn{1}{c}{Data} & \multicolumn{1}{c}{$n$} & $\nnz(K)$ & \multicolumn{1}{c}{$\min_j \abs{K_{jj}}$} & \multicolumn{1}{c}{$\max_j \abs{K_{jj}}$} \\\midrule
    synthetic &       6 &       16 & $8.1 \times 10^{-21\phantom{0}}$  & $3.9 \times 10^{11}$ \\
    DFG       & 1{,}745 &  9{,}677 & $4.3 \times 10^{-159}$ & $3.8 \times 10^{14}$ \\
    WL        & 1{,}759 & 10{,}485 & $1.2 \times 10^{-144}$ & $1.6 \times 10^{16}$ \\\bottomrule
  \end{tabular}
\end{table}

For real-world datasets, we use two chemical reaction networks fetched from the Searching Chemical Action and Network (SCAN) platform~\cite{Kuwahara2023-js}.
The two datasets, denoted by DFG and WL, represent the synthesis of fluoroglycine and Wöhler's urea synthesis, respectively.
Rate constant matrices $K$ are calculated from the free energies of equilibrium and transient states, assuming the canonical ensemble~\eqref{eq:canonical} with $T = \SI{300}{K}$ and $\Gamma = 1$.
We truncate off-diagonal entries $K_{ij}$ and $K_{ji}$ to zero if $\frac{K_{ij}}{\pi_j} = \frac{K_{ji}}{\pi_i} < 10^{-200}$ and eliminate all-zero rows and columns from $K$.
Properties of the resulting rate constant matrices are summarized in \Cref{tab:datasets}.

We execute the RCMC method (\Cref{alg:improved}) of Types~A and~B to compute approximate solutions $q^{(k)} \approx x\prn[\big]{t^{(k)}}$ for $k = 0, \dotsc, n$.
The initial vector $p$ is set to $e_1 = {(1, 0, \dotsc, 0)}^\top$.
Based on the discussion in \Cref{sec:time-computation}, we determine the reference time $t^{(k)}$ in the following three ways:
\begin{itemize}
  \item ``diag'': $t^{(k)} = t_\mathrm{diag}^{(k)}$ defined by~\eqref{def:t-k-diag},
  \item ``eigen'': $t^{(k)} = t_\mathrm{eigen}^{(k)}$ defined by~\eqref{def:t-k-eig},
  \item ``gershgorin'': $t^{(k)} = t_\mathrm{gershgorin}^{(k)} \defeq \ln 2 / \sqrt{\hat{\sigma}\prn[\big]{K_{S^{(k)}S^{(k)}}} \hat{\rho}\prn[\big]{D^{(k)}}}$.
\end{itemize}
In ``eigen,'' the spectral radii are computed by the Lanczos method\footnote{%
  More precisely, for a self-adjoint matrix $A$, we apply the Lanczos method to the ``symmetrized'' matrix $B \coloneqq \Pi^{-\frac12} A \Pi^{\frac12}$ and retrive the eigendecomposition of $A$ from that of $B$; see \Cref{sec:adjointness}.
} using \texttt{Spectra}\footnote{
  \texttt{\url{https://spectralib.org/}}
} version 1.0.1.
Incremental updates for LU decompositions of $K_{SS}$ and $M$ (only for Type~B) given in \Cref{sec:incremantally-updating-lu-decomposition} and \Cref{sec:incremental-update-M}, respectively, are employed.
Matrix operations are performed in sparse representation.

Finally, we evaluate the error~\eqref{def:err}, which we call the \emph{$\pi$-norm error}, for every $k$ and compare them with the upper bounds given in \cref{thm:error-main}.
The eigendecomposition of rate constant matrices, used for calculating error bounds and the exact solution $x\prn{t^{(k)}}$, is pre-computed by the Lanczos method with 200-digit precision floating-point numbers via the \texttt{GNU MPFR} library\footnote{
  \texttt{\url{https://www.mpfr.org/}}
} of version 4.2.0.

\subsection{Results}
\begin{table}[tbp]
  \centering
  \caption{Running time (in seconds) of the RCMC method.}\label{tbl:runnning-time}
  \begin{tabular}{lrrrrrr}\toprule
     & \multicolumn{2}{c}{diag}                      & \multicolumn{2}{c}{eigen}                     & \multicolumn{2}{c}{gershgorin} \\\cmidrule(lr){2-3} \cmidrule(lr){4-5} \cmidrule(lr){6-7}
     & \multicolumn{1}{c}{Type~A} & \multicolumn{1}{c}{Type~B} & \multicolumn{1}{c}{Type~A} & \multicolumn{1}{c}{Type~B} & \multicolumn{1}{c}{Type~A} & \multicolumn{1}{c}{Type~B} \\\midrule
  synthetic & $< \SI{1}{ms}$ & $< \SI{1}{ms}$ & $< \SI{1}{ms}$ & $< \SI{1}{ms}$ & $< \SI{1}{ms}$ & $< \SI{1}{ms}$ \\
  DFG       & 1.28 & 2.59 & 72.24 & 71.48 & 1.85 & 3.08 \\
  WL        & 1.58 & 2.86 & 85.94 & 86.25 & 2.21 & 3.39 \\\bottomrule
  \end{tabular}
\end{table}

\begin{figure}[tbp]
  \centering
  \includegraphics[width=\linewidth]{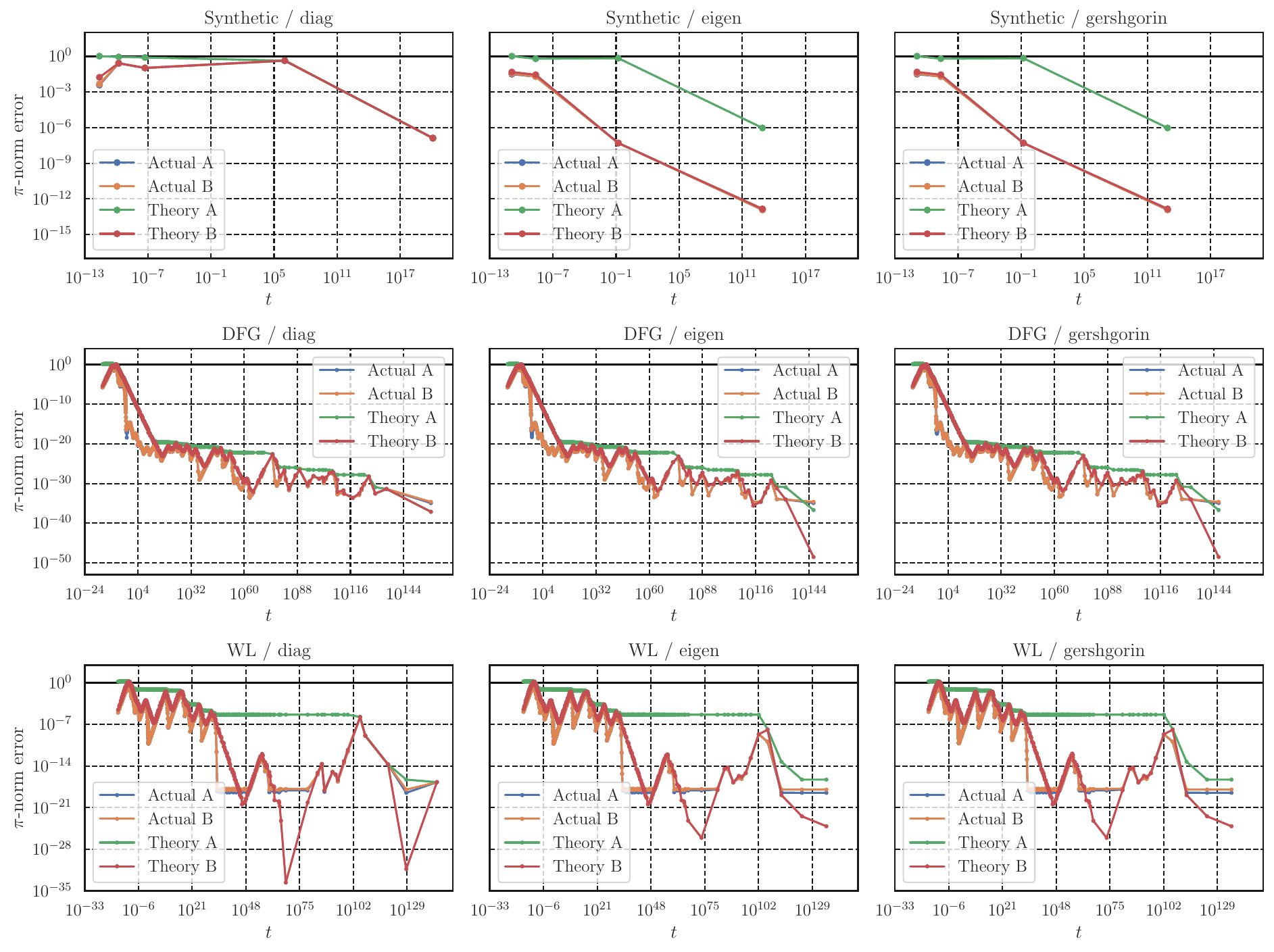}
  \caption{%
    Log-log plots of the actual $\pi$-norm errors in approximate solutions (``Actual A/B'') and their theoretical upper bounds (``Theory A/B''), where A and B indicate the types of the RCMC method.
    The errors are plotted as points $\prn[\big]{t^{(k)}, \Err_{S^{(k)}, t^{(k)}}(p)}$ for each $k$.
    Points with $t^{(k)} = 0$ or $+\infty$ are omitted.
    Actual A and B overlap closely in all plots.
    In the plots for synthetic data, Theory B also aligns closely with Actual A and B.
  }\label{fig:error}
\end{figure}

\Cref{tbl:runnning-time} shows the running time of the RCMC method.
Across all combinations of data, type, and the reference time computation method, the RCMC method ran within 100 seconds.
Notably, ``diag'' and ``gershgorin'' required at most only 4 seconds.

\Cref{fig:error} illustrates the comparison of the actual $\pi$-norm errors and their theoretical upper bounds.
For the synthetic data, the theoretical bounds correctly served as upper bounds on actual errors for all $k$.
For DFG and WL, this was true only for $k$ such that the theoretical bounds for Type~B are larger than $10^{-35}$ and $10^{-18}$, respectively.
This is because the RCMC method is performed with double-precision floating-point numbers (about 15-digit precision); we have confirmed that the errors fell below the theoretical bounds when recalculated with 50-digit precision.
Nevertheless, the approximate solutions for such $k$ computed with double precision align with the exact solutions for the first 12 to 14 digits.
Thus, double precision is sufficient for most applications.
It is worth mentioning that the double precision was insufficient for the eigendecomposition of $K$, even when applied to the synthetic data, as the matrix is extremely ill-conditioned.
This fact highlights an advantage of the RCMC method in terms of numerical stability.

\paragraph{Type~A versus~Type~B.}
In all datasets, the output of the RCMC method of Types~A and~B were almost identical. Notably, the errors of approximate solutions of Type~A were consistently lower than the theoretical bounds not only of Type~A but also of Type~B.
Regarding the computational time, Type~B applied to DFG and WL additionally took about 1 second to perform the incremental update of an LU decomposition of $M$.
Thus, in practice, employing Type~A is sufficient from the viewpoint of running time.

\begin{figure}[tbp]
  \centering
  \includegraphics[width=\linewidth]{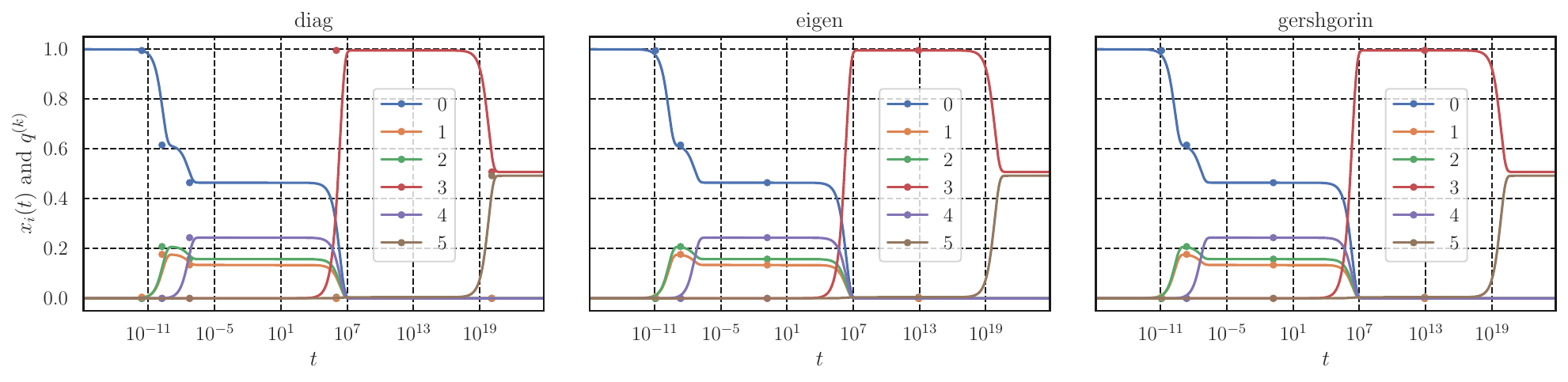}
  \caption{%
    Exact solutions (solid curves) and approximate solutions (points) for the synthetic data.
    Approximate solutions of both Types~A and~B give visually indistinguishable plots.
  }\label{fig:synthetic-pop}
\end{figure}

\paragraph{Comparison of reference time computation methods.}
The output of ``eigen'' and ``gershgorin'' are nearly identical in all datasets.
Thus, given the computational time shown in \Cref{tbl:runnning-time}, we recommend the use of ``gershgorin''.
As for ``diag,'' it maintains a relatively small error for DFG and WL but exhibits larger errors to synthetic data.
For a detailed analysis, we show plots of the exact and approximate solutions for the synthetic data in \cref{fig:synthetic-pop}.
In ``diag,'' $t^{(k)}$ are picked right after significant changes in the exact solution.
In contrast, both ``eigen'' and ``gershgorin'' choose the midpoint (in terms of the logarithmic scale) of periods where the exact solution is nearly constant.
Consequently, these methods yield smaller errors between the approximate and exact solutions.

\paragraph{$L_\infty$-norm error.}

\begin{figure}[tbp]
  \centering
  \includegraphics[width=\linewidth]{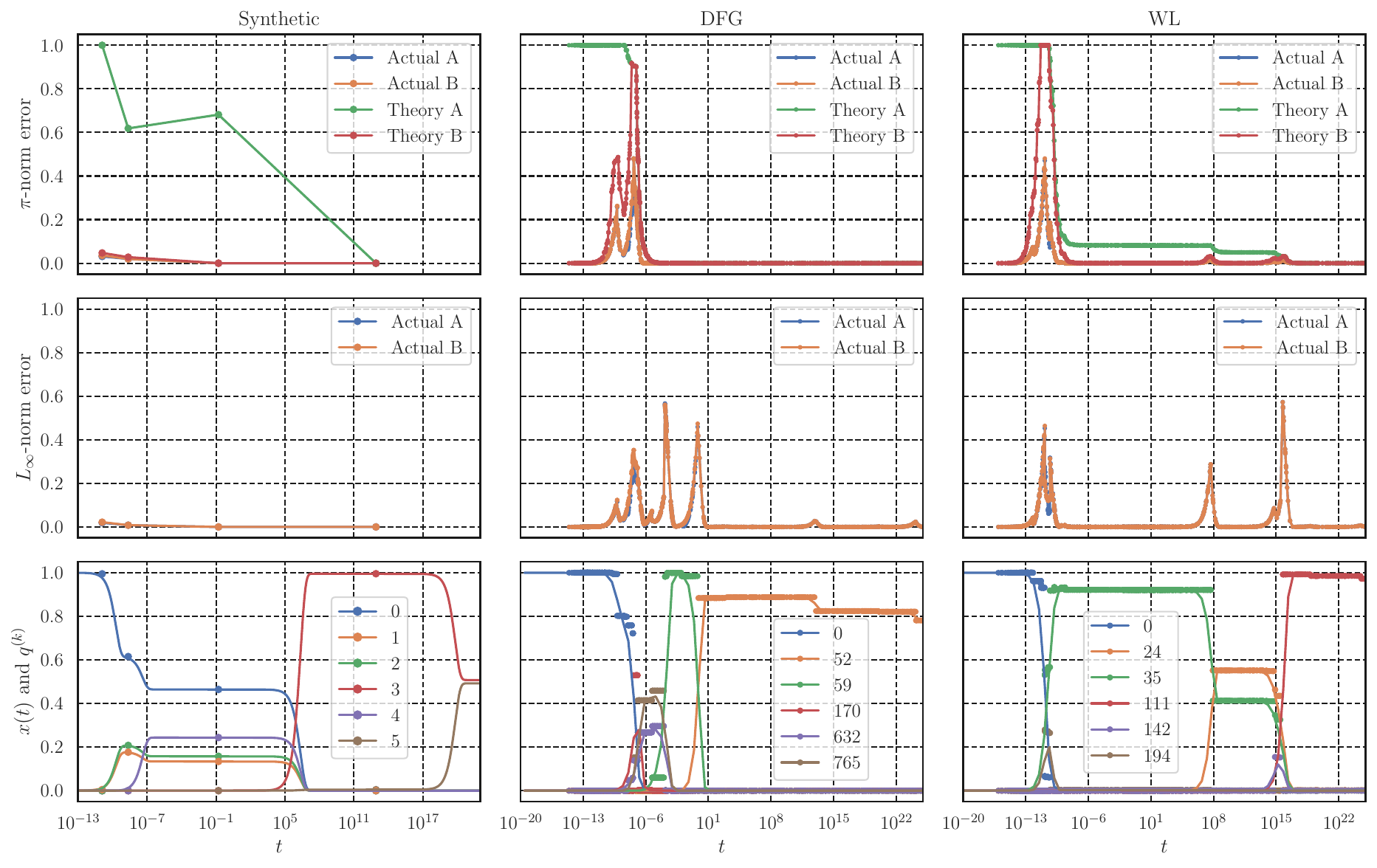}
  \caption{%
    The top and middle rows: log-linear plots of the $\pi$- and $L_\infty$-norm errors, respectively.
    ``Actual A/B'' means the actual errors in approximate solutions, whereas ``Theory A/B'' in the top row indicates the theoretical upper bounds.
    ``A'' and ``B'' refer to the types of the RCMC method.
    Actual A and B overlap closely in all error plots.
    The bottom row: comparison of exact solutions (solid curves) and the approximate solutions (points) computed by the RCMC method of Type~B.
    Only the top six curves are plotted for clarity.
    In all plots, the times $t^{(k)}$ are computed by ``eigen''.
  }\label{fig:linf-error}
\end{figure}

While the $\pi$-norm error~\eqref{def:err} is tractable from the viewpoint of error analysis, it is not commonly used in general contexts.
We illustrate in the top and middle rows of \Cref{fig:linf-error} the log-linear plots of the $\pi$-norm errors and the \emph{$L_\infty$-norm errors}, defined to be $\norm[\big]{q^{(k)} - x(t^{(k)})}_\infty$, respectively, focusing on shorter time intervals (with $t^{(k)}$ computed by the ``eigen'').
Although the $\pi$- and $L_\infty$-norm errors typically follow analogous trajectories in the plots, there are moments when the $\pi$-norm error is small but the $L_\infty$-norm error is large.
To capture the time when a large $L_\infty$-norm error occurs, we plot the approximate and exact solutions in the bottom row of \cref{fig:linf-error}.
For DFG and WL, wherein the RCMC method yields dense outputs of approximate solutions, the dynamics of approximate solutions resemble step functions, while the exact solutions form smooth curves.
This discrepancy leads to a significant error amplification when the exact solution undergoes substantial fluctuations.
Since the RCMC method is based on QSSA, approximating solutions accurately at such transient timings is difficult.
Nonetheless, owing to its capability to perform calculations with notable speed and numerical stability, it is deemed suitable for applications that require fast computations of rough approximations of the solutions to master equations, like the kinetics-based navigation in automatic chemical reaction pathway search techniques~\cite{Ohno2004-iw,Sumiya2020-ti}.

\section{Conclusion}\label{sec:conclusion}
In this paper, we have explored mathematical aspects of the RCMC method.
We have reformulated the RCMC method in terms of matrix computation and conducted theoretical error analysis.
Building on this foundation, we have also discussed its efficient and numerically stable implementations.
Numerical experiments with both synthetic and real-world datasets have demonstrated computational efficiency and numerical stability, and validated the derived error bounds.
Potential future work includes extending the RCMC method to nonlinear master equations, such as those involving bimolecular reactions.

\section*{Acknowledgments}
The authors thank Satoshi Maeda, Yu Harabuchi, and Wataru Matsuoka of Hokkaido University for fruitful discussions.
This work was supported by JST ERATO Grant Number JPMJER1903.

\printbibliography[heading=bibintoc]

\newpage
\appendix
\section*{Appendix}
\section{Vector Spaces with Diagonal Metric}\label{sec:metrix-vector-spaces}
This section summarizes basic results on the real vector space $\R^n$ equipped with the $\pi$-inner product~\eqref{def:pi-inner-product}, where $\pi = {(\pi_i)}_{i \in [n]} \in \Rpp^n$.
We let $\Pi \defeq \diag(\pi)$.

\subsection{Adjointness}\label{sec:adjointness}
The \emph{adjoint matrix} of $A \in \R^{n \times n}$ is the unique matrix $A^* \in \R^{n \times n}$ such that $\pipr{a, Ab} = \pipr{A^*a, b}$ holds for all $a, b \in \R^n$.
The adjoint matrix $A^*$ is explicitly written as $A^* = \Pi \trsp{A} \Pi^{-1}$ by
\begin{align}
  \pipr{a, Ab} = \trsp{a}\Pi^{-1}Ab = \trsp{\prn[\big]{\Pi\trsp{A}\Pi^{-1}a}}\Pi^{-1}b = \piprbig{\Pi\trsp{A}\Pi^{-1}a, b}.
\end{align}
It holds that ${(A_1 A_2)}^* = A_2^* A_1^*$ for $A_1, A_2 \in \R^{n \times n}$ and $\prn{A^{-1}}^* = \prn{A^*}^{-1}$ for nonsingular $A \in \R^{n \times n}$.

A matrix $A \in \R^{n \times n}$ is called \emph{self-adjoint} if $A^* = A$, i.e., $A\Pi = \Pi\trsp{A}$.
If $A$ is self-adjoint, $B \defeq \Pi^{-\frac12} A \Pi^{\frac12}$ is a symmetric matrix.
Here, for a non-negative diagonal matrix $D = \diag\prn{d_1, \dotsc, d_n} \ge O$, we let $D^{\frac12} \defeq \diag\prn{\sqrt{d_1}, \dotsc, \sqrt{d_n}}$.
Let $B = V \Lambda \trsp{V}$ be an eigendecomposition of $B$, where $V$ is an orthogonal matrix (i.e.\ $VV^\top = V^\top V = I_n$) and $\Lambda = \diag(\lambda_1, \dotsc, \lambda_n)$ is a diagonal matrix.
Letting $U \defeq \Pi^{\frac12}V$, we have $A = \Pi^{\frac12}V\Lambda \trsp{V}\Pi^{-\frac12} = U \Lambda U^{-1}$.
Thus, $A$ and $B$ have the same eigenvalues.
In particular, eigenvalues of self-adjoint matrices are real.
See also~\cite[Chapter~IV]{Dodson1991-ic}.

Let $u_1, \dotsc u_n$ be the column vectors of $U$, i.e., $u_k$ is an eigenvector of $A$ corresponding to the eigenvalue $\lambda_k$.
The matrix $U$ satisfies $\trsp{U}\Pi^{-1}U = \trsp{V}V = I_n$, which means that $\pipr{u_k, u_l}$ is $1$ if $k = l$ and $0$ otherwise for $k \in [n]$.
Namely, adjoint matrices have orthonormal eigenbases.
If we expand $a \in \R^n$ with respect to the eigenbasis as $a = \sum_{k=1}^n c_k u_k$ with $c_1, \dotsc, c_n \in \R$, for any $l \in [n]$, it holds that
\begin{align}
  \pipr{u_l, a} = \sum_{k=1}^n c_k \pipr{u_l, u_k} = c_l,
\end{align}
and hence we have $a = \sum_{k=1}^n \pipr{u_k, a} u_k$.

\subsection{Positive Semi-definiteness}\label{sec:positive-semi-definiteness}
A self-adjoint matrix $A \in \R^{n \times n}$ is called \emph{positive semi-definite} if $\pipr{a, Aa} \ge 0$ for $a \in \R^n$ and \emph{positive definite} if $\pipr{a, Aa} > 0$ for $a \in \R^n \setminus \set{\zeros}$.
Similarly, $A$ is called \emph{negative semi-definite} if $-A$ is positive semi-definite and \emph{negative definite} if $-A$ is positive definite.
The definiteness are equivalent to those of $B \defeq \Pi^{-\frac12} A \Pi^{\frac12}$ by
\begin{align}\label{eq:derive-adjoint}
  \pipr{a, Aa} = \trsp{a}\Pi^{-1}Aa = \trsp{a}\Pi^{-\frac12} B \Pi^{-\frac12}a = \trsp{b}Bb,
\end{align}
where $b = \Pi^{-\frac12}a$.
Since $A$ and $B$ have the same eigenvalues, $A$ is positive (resp.\ negative) semi-definite if and only if its eigenvalues are non-negative (resp.\ non-positive) and is positive (resp.\ negative) definite if and only if its eigenvalues are positive (resp.\ negative).

Let $A \in \R^{n \times n}$ be a positive semi-definite matrix with eigendecomposition $U \Lambda U^{-1}$.
The \emph{square root} of $A \in \R^{n \times n}$ is defined as $A^{\frac12} \defeq U \Lambda^{\frac12} U^{-1}$, which is the unique positive semi-definite matrix whose square is $A$.
If $A$ is positive definite, so is $A^{\frac12}$.

For self-adjoint $A, B \in \R^{n \times n}$, let $A \preceq B$ mean that $B - A$ is positive semi-definite.
The binary relation $\preceq$ forms a partial order on the self-adjoint matrices compatible with addition, called the \emph{Löwner order}.
As with the usual case of symmetric matrices, the Löwner order satisfies the following properties.
\begin{itemize}
  \item For nonsingular self-adjoint matrices $A, B$, if $A \preceq B$, then $B^{-1} \preceq A^{-1}$.
  \item $A \preceq B$ implies $P^*AP \preceq P^*BP$ for any $P \in \R^{n \times n}$.
\end{itemize}

\subsection{Coordinate Subspaces}\label{sec:coordinate-subspaces}
Since the metric $\Pi^{-1}$ is dinagonal, we can consider the ``subspace'' of $\R^n$ indexed by $I \subseteq [n]$.
Let $\R^I$ be the vector space having coordinates indexed by $I \subseteq [n]$.
We call $\R^I$ a \emph{coordinate subspace} of $\R^n$.
The space $\R^I$ is naturally regarded as a metric vector space equipped with metric $\Pi_I^{-1}$, where $\Pi_I \defeq \diag(\pi_I)$.
Slightly abusing notation, we also let $\pipr{a, b} \defeq \trsp{a} \Pi_I^{-1} b$ and $\pinorm{a} \defeq \sqrt{\pipr{a, a}}$ for $a, b \in \R^I$.

For $I, J \subseteq [n]$, let $\R^{I \times J}$ denote the set of matrices whose rows and columns are indexed by $I$ and $J$, respectively.
The \emph{adjoint matrix} of $A \in \R^{I \times J}$ is the matrix $A^* \in \R^{J \times I}$ such that $\pipr{a, Ab} = \pipr{A^*a, b}$ for all $a \in \R^I$ and $b \in \R^J$, and is explicitly written as $A^* = \Pi_J \trsp{A} \Pi_I^{-1}$ in the same argument as~\eqref{eq:derive-adjoint}.
For $A \in \R^{I \times I'}$ and $B \in \R^{I' \times I''}$ with $I, I', I'' \subseteq [n]$, it holds that ${(AB)}^* = B^*A^*$.

For $A, B \in \R^{n \times n}$ with $A = B^*$ and $I, J \subseteq [n]$, we have
\begin{align}
  A_{IJ}
  = \prn{B^*}_{IJ}
  = \prn[\big]{\Pi \trsp{B} \Pi^{-1}}_{IJ}
  = \Pi_I \trsp{\prn{B_{JI}}} \Pi_{J}^{-1}
  = \prn{B_{JI}}^*.
\end{align}
In particular, if $A$ is self-adjoint, its principal submatrices are also self-adjoint, and $A_{IJ}$ and $A_{JI}$ are adjoint of each other.
Similarly, the principal submatrices of a positive semi-definite matrix are positive semi-definite.

\subsection{Matrix Norm}\label{sec:matrix-norm}

Let $I, J \subseteq [n]$ and $A \in \R^{I \times J}$.
The \emph{matrix norm}, or the \emph{operator norm}, of $A$ is defined by
\begin{align}\label{def:matrix-norm-appendix}
  \pinorm{A}
  \defeq \max_{a \in \R^J: \pinorm{a} = 1} \pinorm{Aa}
  = \max_{a \in \R^J \setminus \set{\zeros}} \frac{\pinorm{Aa}}{\pinorm{a}}.
\end{align}
Note that different norms are used in the numerator and the denominator in~\eqref{def:matrix-norm-appendix} unless $I = J$.
Letting $b \defeq \Pi_J^{-\frac12}a$ and $B \defeq \Pi_I^{-\frac12} A \Pi_J^{\frac12}$, we have
\begin{align}\label{eq:matrix-norm-singular-value}
  \pinorm{A}
  = \max_{b \in \R^J \setminus \set{\zeros}} \frac{\norm[\big]{Bb}_2}{\norm{b}_2}
  = \sqrt{\rho\prn[\big]{B^\top B}},
\end{align}
where $\rho(\cdot)$ denotes the spectral radius.
Indeed, the spectral radii of $B^\top B$ and $A^* A$ are the same by
\begin{align}
  \rho\prn[\big]{B^\top B}
  = \rho\prn[\big]{\Pi_J^{\frac12} A^\top \Pi_I^{-1} A \Pi_J^{\frac12}}
  = \rho\prn[\big]{\Pi_J^{-\frac12} A^* A \Pi_J^{\frac12}}
  = \rho\prn[\big]{A^* A},
\end{align}
meaning that $\pinorm{A} = \sqrt{\rho(A^*A)} = \sqrt{\rho(AA^*)}$.
In particular, if $I = J$ and $A$ is self-adjoint, $\pinorm{A} = \rho(A)$ holds.

\section{RCMC Method in the Original Form}\label{sec:original-rcmc}
In this section, we describe the RCMC method in the original form given by Sumiya et al.~\cite{Sumiya2015-br} and discuss the correspondence between the original and our formulations.

\subsection{Original Description}
\begin{algorithm}[tb]
	\caption{RCMC method by Sumiya et al.~\cite{Sumiya2015-br,Sumiya2017-qo}}\label{alg:original}
	\begin{algorithmic}[1]
    \Input a rate constant matrix $K \in \R^{n \times n}$ with stationary distribution $\pi \in \Delta_n^\circ$, an initial value $p \in \Delta_n$, and an end time $t_{\mathrm{max}} \in \Rp$
    \Output $t^{(0)}, t^{(1)}, \dotsc \in \Rp$ and $q^{(0)}, q^{(1)}, \dotsc \in \R^n$
    \State Output $t^{(0)} \defeq 0$ and $q^{(0)} \defeq p$
    \State $S^{(0)} \gets \emptyset$, $T^{(0)} \gets [n]$, $\Omega^{(0)} \gets I_n$, $p^{(0)} \gets p$, $K^{(0)} \gets K$
    \For{$k = 1, 2, \dotsc$}
      \State Take $\prn[\big]{i^{(k)}, j^{(k)}} \in \argmax\set[\big]{K^{(k-1)}_{ij}}[(i, j) \in T^{(k-1)} \times T^{(k-1)}, i \ne j]$ {\label{line:original-argmax}}
      \State $t^{(k)} \defeq \frac{1}{K^{(k-1)}_{i^{(k)}j^{(k)}}}$ {\label{line:original-t}}
      \If{$t^{(k)} > t_{\mathrm{max}}$}
        \State \textbf{break}
      \EndIf
      \State $S^{(k)} \defeq S^{(k-1)} \cup \set[\big]{j^{(k)}}$ and $T^{(k)} \defeq T^{(k-1)} \setminus \set[\big]{j^{(k)}}$
      \State $\Omega_{ij}^{(k)} \defeq \Omega_{ij}^{(k-1)} + \sigma^{(k)} K^{(k-1)}_{ij^{(k)}} \Omega_{j^{(k)}j}^{(k-1)}$ \quad $\prn[\big]{i \in T^{(k)}, j \in [n]}$ {\label{line:original-omega}} \Comment{$\sigma^{(k)}$ is~\eqref{def:sigma}}
      \State $p^{(k)}_i \defeq p^{(k-1)}_i + \sigma^{(k)} K^{(k-1)}_{ij^{(k)}} p^{(k-1)}_{j^{(k)}}$ \quad $\prn[\big]{i \in T^{(k)}}$ {\label{line:original-p}}
      \State Output $t^{(k)}$ and $\displaystyle q^{(k)} = \prn{\sum_{i \in T^{(k)}} p_i^{(k)} \frac{\Omega^{(k)}_{ij}\pi_j}{\sum_{l=1}^n \Omega^{(k)}_{il}\pi_l}}_{j \in [n]}$ {\label{line:original-q}}
      \State $\displaystyle K^{(k)}_{ij} \defeq \frac{K^{(k-1)}_{ij} + \sigma^{(k)} K^{(k-1)}_{i j^{(k)}} K^{(k-1)}_{j^{(k)} j}}{1 + \sigma^{(k)} K^{(k-1)}_{j^{(k)} j}}$ \quad $\prn[\big]{i,j \in T^{(k)}}$ {\label{line:original-K}}
    \EndFor
	\end{algorithmic}
\end{algorithm}

\Cref{alg:original} describes the original procedure of the RCMC method presented in~\cite{Sumiya2017-qo}.
As with \cref{alg:improved}, given a rate constant matrix $K \in \R^{n \times n}$, an initial value $p \in \R^n$, and the end time $t_{\mathrm{max}} \in \Rp$ as input, \cref{alg:original} outputs a sequence of reals $0 = t^{(0)} < t^{(1)} < t^{(2)} < \dotsb \le t_{\mathrm{max}}$ and vectors $q^{(0)}, q^{(1)}, \dotsc \in \R^n$ such that $q^{(k)}$ approximates the solution $x\prn[\big]{t^{(k)}}$ of the master equation~\eqref{def:master} with initial value $x(0) = p$.

\Cref{alg:original} keeps a growing set $S^{(k)} \subseteq [n]$ and its complement $T^{(k)} = [n] \setminus S^{(k)}$ as \cref{alg:improved}.
\Cref{alg:original} additionally updates two matrices $K^{(k)} \in \R^{T^{(k)} \times T^{(k)}}$, $\Omega^{(k)} \in \R^{T^{(k)} \times n}$, and a vector $p^{(k)} \in \R^{T^{(k)}}$.
In the original description, states in $S^{(k)}$ are called ``steady'' or ``contracted,'' and states in $T^{(k)}$ are called \emph{super-states} in a sense that each super-state $i \in T^{(k)}$ is regarded as a ``mixture'' with some portions of the original states in $[n]$.
The matrix $K^{(k)}$ is a rate constant matrix\footnote{%
  We will prove in \Cref{lem:original-K-schur} that $K^{(k)}$ is indeed a rate constant matrix in the sense that it satisfies~\ref{item:rcm1}--\ref{item:rcm3}. This fact is, however, rather nontrivial at this point.}
on the super-states, whose $(i, j)$ entry is the rate constant from super-state $j$ to super-state $i$ for $i \ne j$.
The $(i, j)$ entry of $\Omega^{(k)}$ represents the contribution of original state $j$ to super-state $i$.
The $i$th component of $p^{(k)}$ represents the current approximate \emph{population} of super-state $i$, where ``population'' in this context means the same thing as quantity or probability.

\Cref{alg:original} initializes $S^{(0)}$ as the empty set, $K^{(0)}$ as a given rate constant matrix $K$, $\Omega^{(0)}$ as the identity matrix, and $p^{(0)}$ as a given initial vector $p$.
In each $k$th iteration, \cref{alg:original} first finds a maximum off-diagonal $K^{(k-1)}_{i^{(k)}j^{(k)}}$ of $K^{(k-1)}$ with $i^{(k)}, j^{(k)} \in T^{(k-1)}$.
If the current time $t^{(k)}$, which is calculated as the inverse of $K^{(k-1)}_{i^{(k)}j^{(k)}}$, is larger than $t_{\mathrm{max}}$, the algorithm halts.
Otherwise, $j^{(k)}$ is recognized as a new steady state and moved from $T$ to $S$.
The algorithm then updates $\Omega^{(k)}$, $p^{(k)}$, $q^{(k)}$, and $K^{(k)}$ according to the equations given in \Cref{line:original-omega,line:original-p,line:original-q,line:original-K}.
These formulas are intuitively explained in \cite{Sumiya2020-ti,Sumiya2017-qo} as follows: state $j^{(k)}$ is ``distributed'' to every remaining super-state $i \in T^{(k+1)}$ with the ratio of $\sigma^{(k)} K^{(k-1)}_{ij^{(k)}}$, where
\begin{align}\label{def:sigma}
  \sigma^{(k)} \defeq \frac{1}{\sum_{i \in T} K^{(k-1)}_{ij^{(k)}}}.
\end{align}
The update formulas of $\Omega^{(k)}$ and $p^{(k)}$ at \Cref{line:original-omega,line:original-p} represent this distribution.
\Cref{line:original-q} converts the super-states' population $p^{(k)}$ at time $t^{(k)}$ into that of the original states, where the formula is derived under the assumption that the ingredient states in every super-state are in equilibrium.
Finally, at \Cref{line:original-K}, the QSSA update with a bipartition $\set[\big]{\set[\big]{j^{(k)}}, T^{(k)}}$ of $T^{(k-1)}$ (see \Cref{sec:qssa}) as well as the division by $1 + \sigma K^{(k-1)}_{j^{(k)} j}$ are applied to the rate constant matrix $K^{(k-1)}$.
This division is similar to the formula used in the lumping~\cite{Kuo1969-vh,Wei1969-kr} to change the stationary distribution associated with the rate constant matrix.

\subsection{Reformulation}\label{sec:reformulation}
In this section, we reformulate \cref{alg:original} to obtain \cref{alg:improved} of Type~A.
Let $r$ be the number of iterations of \Cref{alg:original}, i.e., $t^{(k)} \le t_{\mathrm{max}}$ for $k \in [r]$ and $t^{(k)} > t_{\mathrm{max}}$ for $k = r+1$.

\begin{lemma}\label{lem:original-K-schur}
  For $k = 0, \dotsc, r$, the following statements hold.
  \begin{enumerate}
    \item If $k \ge 1$, then we have $K_{j^{(k)}j^{(k)}}^{(k-1)} \ne 0$ and $\sigma^{(k)} = \frac{1}{K_{j^{(k)}j^{(k)}}^{(k-1)}}$.\label{item:original-K-schur-1}
    \item $K_{S^{(k)}S^{(k)}}$ is nonsingular and it holds that $K^{(k)} = D^{(k)}\diag\prn[\big]{s^{(k)}}$, where
    \begin{align}\label{def:D-k}
      D^{(k)} \defeq K_{T^{(k)}T^{(k)}} - K_{T^{(k)}S^{(k)}}K_{S^{(k)}S^{(k)}}^{-1}K_{S^{(k)}T^{(k)}}
    \end{align}
    is the Schur complement of $K_{S^{(k)}S^{(k)}}$ in $K$ and $s^{(k)} \in \Rpp^{T^{(k)}}$ is some positive vector.\label{item:original-K-schur-2}
  \end{enumerate}
\end{lemma}

\begin{proof}
  We show the claims by induction on the number $k$ of iterations.
  When $k = 0$, the claims trivially hold with $s = \ones_n$.
  Suppose that the claims are true for $k - 1$ and consider the case of $k \in [r]$.
  Set $S = S^{(k)}$, $T = T^{(k)}$, and $j = j^{(k)}$.
  By the induction hypothesis, we have $K^{(k-1)} = D^{(k-1)}\diag\prn[\big]{s^{(k-1)}}$, which is a rate constant matrix (i.e.~\ref{item:rcm1}--\ref{item:rcm3} are satisfied) with stationary distribution $\prn{\frac{\pi_i}{s_i^{(k-1)}}}_{i \in T^{(k-1)}}$, where $\pi$ is the stationary distribution associated with $K$.

  We first show~\ref{item:original-K-schur-1}.
  By the definition of $r$, $K^{(k-1)}_{i^{(k)}j}$ is non-zero, which implies $K^{(k-1)}_{jj} \ne 0$ from~\ref{item:rcm1} and~\ref{item:rcm2}.
  We have $\sigma^{(k)} = \frac{1}{K_{jj}^{(k-1)}}$ by~\eqref{def:sigma} and~\ref{item:rcm2}.

  We next show~\ref{item:original-K-schur-2}.
  Since $S = S^{(k-1)} \cup \set{j}$, $K_{SS}$ is nonsingular if and only if $D^{(k-1)}_{jj} \ne 0$, which is true by $D^{(k-1)}_{jj} = \frac{K^{(k-1)}_{jj}}{s^{(k-1)}_{j}} \ne 0$.
  The update formula of $K^{(k)}$ at \Cref{line:original-K} can be expressed as
  \begin{align}
    K^{(k)}
    &= \prn{K_{TT}^{(k-1)} + \sigma^{(k)} K^{(k-1)}_{Tj}K^{(k-1)}_{jT}} {\diag\prn[\Big]{\onestr_T + \sigma^{(k)} K^{(k-1)}_{jT}}}^{-1} \\
    &= \prn{K_{TT}^{(k-1)} - \frac{K^{(k-1)}_{Tj}K^{(k-1)}_{jT}}{K^{(k-1)}_{jj}}} {\diag\prn{\onestr_T - \frac{K^{(k-1)}_{jT}}{K^{(k-1)}_{jj}}}}^{-1} \\
    &= \prn{D^{(k-1)}_{TT} - \frac{D^{(k-1)}_{Tj}D^{(k-1)}_{jT}}{D^{(k-1)}_{jj}}} \diag\prn[\big]{s^{(k-1)}_T} {\diag\prn{\onestr_T - \frac{D^{(k-1)}_{jT}}{D^{(k-1)}_{jj}}}}^{-1},
  \end{align}
  where we used the induction hypothesis $K^{(k-1)} = D^{(k-1)}\diag\prn[\big]{s^{(k-1)}}$ in the last equality.
  Here, $D^{(k-1)}_{TT} - \frac{D^{(k-1)}_{Tj}D^{(k-1)}_{jT}}{D^{(k-1)}_{jj}}$ is nothing but $D^{(k)}$.
  Therefore, letting $s^{(k)}_i \defeq \frac{s^{(k-1)}_i}{1 - D^{(k-1)}_{j^{(k)}i} / D^{(k-1)}_{jj}}$ for $i \in T$, we obtain the desired form of $K^{(k)}$.
\end{proof}

As a corollary of \Cref{lem:original-K-schur}, we obtain $r \le \rank K$.
In addition, $r = \rank K$ if $t_{\mathrm{max}}$ is sufficiently large.
Since $D^{(k)}$ is a rate constant matrix, so is $K^{(k)}$.

Next, we give explicit formulas of $\Omega^{(k)}$ and $p^{(k)}$.

\begin{lemma}\label{lem:original-omega-p}
  For $k = 0, \dotsc, r$, we have
  \begin{align}
    \Omega^{(k)} &= \begin{pmatrix} -K_{TS} K_{SS}^{-1} & I_T\end{pmatrix},\label{def:original-omega} \\
    p^{(k)} &= \Omega^{(k)}p = p_T - K_{TS}K_{SS}^{-1}p_S,\label{eq:original-p-matrix}
  \end{align}
  where $S = S^{(k)}$ and $T = T^{(k)}$.
\end{lemma}

\begin{proof}
  It is easy to see that $p^{(k)} = \Omega^{(k)}p$ holds from \Cref{line:original-omega,line:original-p}.
  Thus, it suffices to show~\eqref{def:original-omega}.
  Set $j = j^{(k)}$.
  \Cref{line:original-omega} is expressed in a matrix form as
  \begin{align}\label{eq:original-omega-lem-1}
    \Omega^{(k)}
    = \Omega^{(k-1)}_{T[n]} + \sigma^{(k)} K^{(k-1)}_{Tj} \Omega^{(k-1)}_{j[n]}
    = \begin{pmatrix} \sigma^{(k)} K^{(k-1)}_{Tj} & I_T \end{pmatrix} \Omega^{(k-1)}.
  \end{align}
  By \Cref{lem:original-K-schur}, it holds that
  \begin{align}\label{eq:original-omega-lem-2}
    \sigma^{(k)} K^{(k-1)}_{Tj}
    = -\frac{K^{(k-1)}_{Tj}}{K^{(k-1)}_{jj}}
    = -\frac{D^{(k-1)}_{Tj}s^{(k)}_j}{D^{(k-1)}_{jj}s^{(k)}_j}
    = -\frac{D^{(k-1)}_{Tj}}{D^{(k-1)}_{jj}},
  \end{align}
  where $s^{(k)}$ is the vector given in \Cref{lem:original-K-schur} and $D^{(k-1)}$ is the Schur complement defined by~\eqref{def:D-k}.
  Define $E^{(k)} \in \R^{T \times T^{(k-1)}}$ by
  \begin{align}\retainlabel{def:P}
    E^{(k)} \defeq \begin{pNiceMatrix}[first-row,last-col]
      j & T & \\
      -\frac{D^{(k-1)}_{Tj}}{D^{(k-1)}_{jj}} & I_T & T
    \end{pNiceMatrix}.
  \end{align}
  From~\eqref{eq:original-omega-lem-1} and~\eqref{eq:original-omega-lem-2}, we have
  \begin{align}\label{eq:original-omega-lem-3}
    \Omega^{(k)}
    = E^{(k)}\Omega^{(k-1)}
    = E^{(k)}E^{(k-1)}\Omega^{(k-2)}
    = \dotsb
    = E^{(k)} E^{(k-1)} \dotsm E^{(1)}.
  \end{align}

  Recall the procedure of the LU decomposition on $K = D^{(0)}$.
  In its first iteration, we eliminate the first column by using the top-left entry as follows:
  \begin{align}\label{eq:original-omega-P}
    L^{(1)}K = \begin{pmatrix}
      *      & * \\
      \zeros & D^{(1)}
    \end{pmatrix},
    \quad\text{where}\quad
    L^{(1)} \defeq \begin{pmatrix}
      1 & \zerostr \\
      -\frac{D^{(0)}_{T^{(1)}j^{(1)}}}{D^{(0)}_{j^{(1)}j^{(1)}}} & I_{T^{(1)}}
    \end{pmatrix} = \begin{pNiceMatrix}[margin]
      1 & \zerostr \\\hline
      \Block{1-2}{E^{(1)}}
    \end{pNiceMatrix}.
  \end{align}
  Here, $*$ denotes some matrix.
  Ignoring the first row in~\eqref{eq:original-omega-P}, we get $E^{(1)}K = \begin{pmatrix} \zeros & D^{(1)} \end{pmatrix}$.
  Similarly, in the next iteration, we have $E^{(2)}E^{(1)}K = \begin{pmatrix} \zeros & \zeros & D^{(2)} \end{pmatrix}$.
  For general $k$, we obtain
  \begin{align}\label{eq:original-omega-PPP}
    E^{(k)} \dotsm E^{(1)}K = \begin{pNiceMatrix}[first-row,last-col=3]
      S & T \\
      O & D^{(k)} & T
    \end{pNiceMatrix}.
  \end{align}
  Comparing $K = \begin{pmatrix} K_{SS} & K_{ST} \\ K_{TS} & K_{TT} \end{pmatrix}$ with~\eqref{eq:original-omega-PPP}, we get
  \begin{align}
    E^{(k)} \dotsm E^{(1)} = \begin{pmatrix} -K_{TS} K_{SS}^{-1} & I_T \end{pmatrix},
  \end{align}
  which is nothing but $\Omega^{(k)}$ by~\eqref{eq:original-omega-lem-3}.
\end{proof}

By \cref{lem:original-omega-p}, the matrix $\Omega^{(k)}$ turns out to be the same as $\Omega$ defined in~\eqref{def:omega}.
Note that~\eqref{def:original-omega} and~\eqref{eq:original-p-matrix} do not depend on the vector $s^{(k)}$, which represents the effect of division by $1 + \sigma^{(k)} K^{(k-1)}_{j^{(k)} j}$ at \Cref{line:original-K}, due to the reduction of $s^{(k)}_j$ in~\eqref{eq:original-omega-lem-2}.
Since the output $q^{(k)}$ is determined from $p^{(k)}$, $\Omega^{(k)}$, and $\pi$, the division does not affect $q^{(k)}$ except for the selection of $j^{(k)}$ at \Cref{line:original-argmax}.

Finally, we establish the following lemma, which states that the original RCMC method is essentially equivalent to the improved one of Type~A.

\begin{theorem}\label{thm:original-equivalence}
  Let $K$ be a rate constant matrix with a stationary distribution $\pi$ and $p \in \R^n$ an initial value.
  Then, it holds that $q^{(k)} = Vp$ for every $k$, where $V$ is the matrix~\eqref{def:V} of Type~A with respect to the bipartition $\set{S^{(k)}, T^{(k)}}$.
\end{theorem}

\begin{proof}
  Let $S = S^{(k)}, T = T^{(k)}$, and $\Omega = \Omega^{(k)}$.
  The formula of computing $q$ at \Cref{line:original-q} can be written as
  \begin{align}\label{eq:original-q-1}
    q^{(k)} = \Pi \trsp{\Omega} {\diag\prn{\Omega\pi}}^{-1} p^{(k)}.
  \end{align}
  We first rewrite the diagonal matrix in~\eqref{eq:original-q-1}.
  It follows from $K\pi = \zeros$ that $K_{SS}\pi_S + K_{ST}\pi_T = \zeros$ and $\pi_S = -K_{SS}^{-1}K_{ST}\pi_T$ hold.
  Letting $M = \Omega \Omega^*$ be the matrix defined by~\eqref{def:M}, we have
  \begin{gather}
    \Omega \pi
    = \pi_T -K_{TS}K_{SS}^{-1}\pi_S
    = \pi_T + K_{TS}K_{SS}^{-2}K_{ST}\pi_T
    = M \pi_T
  \shortintertext{and}
    \diag\prn{\Omega\pi}
    = \diag\prn{M\pi_T}
    = \diag\prn[\big]{\pi_T^\top M^\top}
    = \diag\prn[\big]{\onestr_T M\Pi_T}
    = V_{TT} \Pi_T,\label{eq:original-equivalence-diag}
  \end{gather}
  where the self-adjointness of $M$ is used in the third equality of~\eqref{eq:original-equivalence-diag}.
  Substituting~\eqref{eq:original-p-matrix},~\eqref{eq:original-equivalence-diag}, and $\Omega^\top = \Pi^{-1} \Omega^* \Pi_T$ into~\eqref{eq:original-q-1}, we obtain $q^{(k)} = \Omega^* V_{TT} \Omega p^{(k)} = Vp^{(k)}$ as required.
\end{proof}

A minor difference between \cref{alg:original} and \cref{alg:improved} left is the following: \cref{alg:original} chooses the next steady state $j^{(k)}$ by finding a maximum off-diagonal entry in $K^{(k-1)}$, while \cref{alg:improved} maximizes a diagonal entry of $D^{(k-1)}$ in absolute value.
As stated in \cref{lem:original-K-schur}, the matrices $K^{(k-1)}$ and $D^{(k-1)}$ are in the relation $K^{(k-1)} = D^{(k-1)}\diag\prn[\big]{s^{(k-1)}}$ for some positive vector $s^{(k-1)}$.
We could not find any mathematical rationale for multiplying the diagonal matrix to the Schur complement $D^{(k-1)}$.
Furthermore, finding a maximum diagonal in absolute value is a widely used strategy in selecting steady states in QSSA~\cite{Turanyi2014-ky}, while finding a maximum off-diagonal is not.

\section{Projection onto Probability Simplex}\label{sec:projection}

We present how to compute the projection $\proj(w) = \argmin \set{\pinorm{q - w}}[q \in \Delta_n]$, which appears in the RCMC method of Type~B.
The projection algorithm provided below is a slight extension of an existing one \cite{Duchi2008-at}, which focuses on the case with $\Pi = I_n$.

Note that the projection can be computed by solving the following problem:
\[
  \mathop{\mathrm{minimize}}\quad
  \frac12(q - w)^\top \Pi^{-1} (q - w)
  \qquad
  \mathop{\mathrm{sujbect\ to}}\quad
  q \ge \zeros,\ \onestr_n q = 1.
\]
The optimal solution must satisfy the following Karush--Kuhn--Tucker (KKT) condition:
\begin{align}
  \pi_i^{-1}(q_i - w_i) - \beta_i - \mu &= 0 & \text{for $i \in [n]$}, \label{eq:kkt-lag} \\
  q_i &\ge 0 & \text{for $i \in [n]$}, \label{eq:kkt-ineq} \\
  \beta_i &\ge 0 & \text{for $i \in [n]$}, \label{eq:kkt-beta} \\
  q_i \beta_i &= 0 & \text{for $i \in [n]$}, \label{eq:kkt-comp} \\
  \sum_{i = 1}^n q_i &= 1, \label{eq:kkt-eq}
\end{align}
where $\mu \in \R$ and $\beta_1,\dots,\beta_n \in \Rp$ are the Lagrange multipliers of the equality and inequality constraints, respectively.
Note that the KKT condition is also sufficient for optimality since the optimization problem satisfies Slater's condition and consists of differentiable convex functions (see, e.g., \cite[Section 5.5.3]{Boyd2004-jl}).
Thus, we below aim to compute $q$ that satisfies the KKT condition.

Without loss of generality, we assume that $w_1,\dots,w_n$ are sorted to satisfy
\[
  \pi_1^{-1} w_1 \ge \dots \ge \pi_n^{-1} w_n.
\]
From \eqref{eq:kkt-ineq}, each $q_i$ satisfies either $q_i = 0$ or $q_i > 0$, where the latter implies $\beta_i = 0$ and $\pi_i^{-1}q_i = \pi_i^{-1}w_i + \mu$ due to \eqref{eq:kkt-lag} and \eqref{eq:kkt-comp}.
Therefore, the above sorting of $w_1,\dots,w_n$ leads to
\[
  \pi_1^{-1} q_1 \ge \dots \ge \pi_{\ell}^{-1} q_{\ell} > \pi_{\ell+1}^{-1} q_{\ell+1} = \dots = \pi_n^{-1} q_n = 0,
\]
where $\ell$ is the largest index such that $q_\ell > 0$.
Once we find such $\ell$, \eqref{eq:kkt-eq} implies $\mu = \frac{1 - \sum_{i=1}^\ell w_i}{\sum_{i=1}^\ell \pi_i}$, and the optimal solution $q$ can be written with such $\mu$ as
\begin{equation}\label{eq:r-q-mu}
  q_j = \begin{cases*}
    w_j + \pi_j \mu & for $j \le \ell$, \\
    0 & for $j > \ell$.
  \end{cases*}
\end{equation}
In what follows, we show that such $\ell$ can be computed as follows:
\[
  \ell = \max\set{j \in [n]}[w_j + \pi_j \frac{1 - \sum_{i=1}^j w_i}{\sum_{i=1}^j \pi_i} \eqqcolon t_j > 0].
\]
In other words, we can compute $\ell$ by checking the sign of $t_j$ in increasing order of $j$.
We prove this claim by showing $t_j > 0$ for $j \le \ell$ and $t_j \le 0$ for $j > \ell$.
In the following analysis, we repetitively use $1 - \sum_{i=1}^\ell w_i = \mu \sum_{i=1}^\ell \pi_i$, which is obtained from \eqref{eq:kkt-eq} and \eqref{eq:r-q-mu}.

For $j = \ell$, we have
\[
  w_\ell + \pi_\ell \frac{1 - \sum_{i=1}^\ell w_i}{\sum_{i=1}^\ell \pi_i}
  =
  w_\ell + \pi_\ell \mu
  =
  q_\ell
  > 0.
\]

For $j < \ell$, we have
\begin{align}
  w_j + \pi_j \frac{1 - \sum_{i=1}^j w_i}{\sum_{i=1}^j \pi_i}
  &=
  \frac{\pi_j}{\sum_{i=1}^j \pi_i} \prn*{\frac{w_j}{\pi_j}\sum_{i=1}^j \pi_i + 1 - \sum_{i=1}^j w_i}
  \\
  &=
  \frac{\pi_j}{\sum_{i=1}^j \pi_i} \prn*{\frac{w_j}{\pi_j}\sum_{i=1}^j \pi_i + 1 - \sum_{i=1}^\ell w_i + \sum_{i=j+1}^{\ell} w_i}
  \\
  &=
  \frac{\pi_j}{\sum_{i=1}^j \pi_i} \prn*{\frac{w_j}{\pi_j}\sum_{i=1}^j \pi_i + \mu\sum_{i=1}^\ell\pi_i + \sum_{i=j+1}^{\ell} w_i}
  \\
  &=
  \frac{\pi_j}{\sum_{i=1}^j \pi_i} \prn*{\frac{w_j}{\pi_j}\sum_{i=1}^j \pi_i + \mu\sum_{i=1}^j\pi_i + \mu\sum_{i=j+1}^\ell\pi_i + \sum_{i=j+1}^{\ell} w_i}
  \\
  &=
  \frac{\pi_j}{\sum_{i=1}^j \pi_i} \prn*{\frac{w_j + \pi_j\mu}{\pi_j}\sum_{i=1}^j \pi_i + \sum_{i=j+1}^{\ell}(w_i + \pi_i \mu)}.
\end{align}
The right-hand side is positive since $w_i + \pi_i \mu = q_i > 0$ for $i \le \ell$.

For $j > \ell$, we have
\begin{align}
  w_j + \pi_j \frac{1 - \sum_{i=1}^j w_i}{\sum_{i=1}^j \pi_i}
  &=
  \frac{\pi_j}{\sum_{i=1}^j \pi_i} \prn*{\frac{w_j}{\pi_j}\sum_{i=1}^j \pi_i + 1 - \sum_{i=1}^j w_i}
  \\
  &=
  \frac{\pi_j}{\sum_{i=1}^j \pi_i} \prn*{\frac{w_j}{\pi_j}\sum_{i=1}^j \pi_i + 1 - \sum_{i=1}^\ell w_i - \sum_{i=\ell+1}^{j} w_i}
  \\
  &=
  \frac{\pi_j}{\sum_{i=1}^j \pi_i} \prn*{\frac{w_j}{\pi_j}\sum_{i=1}^j \pi_i + \mu\sum_{i=1}^\ell\pi_i - \sum_{i=\ell+1}^{j} w_i}
  \\
  &=
  \frac{\pi_j}{\sum_{i=1}^j \pi_i} \prn*{\frac{w_j}{\pi_j}\sum_{i=1}^\ell \pi_i + \mu\sum_{i=1}^\ell\pi_i + \frac{w_j}{\pi_j}\sum_{i=\ell+1}^j\pi_i- \sum_{i=\ell+1}^{j} w_i}
  \\
  &=
  \frac{\pi_j}{\sum_{i=1}^j \pi_i} \prn*{\frac{w_j + \pi_j\mu}{\pi_j}\sum_{i=1}^\ell \pi_i + \sum_{i=\ell+1}^j\pi_i (\pi_j^{-1}w_j - \pi_i^{-1}w_i)}
  .
\end{align}
The right-hand side is non-positive since $w_j + \pi_j\mu = -\pi_i\beta_i \le 0$ for $j > \ell$ due to \eqref{eq:kkt-lag}, \eqref{eq:kkt-beta}, and \eqref{eq:r-q-mu} and $\pi_j^{-1}w_j \le \pi_i^{-1}w_i$ for $j \ge i$ by the sorting.

\begin{algorithm}[tb]
	\caption{$\pi$-Norm projection of $w \in \R^n$ onto $\Delta_n$}\label{alg:projection}
	\begin{algorithmic}[1]
    \State Sort the components of $w$ to satisfy $\pi_1^{-1} w_1 \ge \dots \ge \pi_n^{-1} w_n$ {\label{line:sort-pi-w}}
    \State Find $\ell \defeq \max\set{j \in [n]}[w_j + \pi_j \frac{1 - \sum_{i=1}^j w_i}{\sum_{i=1}^j \pi_i} > 0]$
    \State $\mu \defeq \frac{1 - \sum_{i=1}^\ell w_i}{\sum_{i=1}^\ell \pi_i}$
    \State Output $q_i \defeq \max\set{w_i + \pi_i \mu, 0}$ for $i \in [n]$
	\end{algorithmic}
\end{algorithm}

To conclude, the desired projection $q$ can be computed as \Cref{alg:projection}.
As for the computation cost, \Cref{line:sort-pi-w} is the dominant part, which takes $\mathrm{O}(n \log n)$ time.

\section{Determining the Optimal Reference Time}\label{sec:determining-optimal-time}

In this section, we consider the problem of minimizing
\begin{align}
  f(t) \defeq \max_{\lambda < 0} \min \set{
    -\frac{a}{\lambda} + \e^{t\lambda},
    -\frac{\lambda}{b} + 1 - \e^{t\lambda}
  }
\end{align}
over $t \in \Rp$ with $a, b \in \Rpp$.
In the setting of~\eqref{eq:minimize-error-problem}, $a = \rho(D)$ and $b = \sigma(K_{SS})$.

For fixed $t \in \Rp$, let $\alpha_t(\lambda) \defeq -\frac{a}{\lambda} + \e^{t\lambda}$ and $\beta_t(\lambda) \defeq -\frac{\lambda}{b} + 1 - \e^{t\lambda}$.
Then, $\alpha_t$ and $\beta_t$ are are strictly monotone increasing and decreasing, respectively, and they satisfy
\begin{align}
  \lim_{\lambda \to -\infty} \alpha_t(\lambda) &= \begin{cases}
    1 & (t = 0),\\
    0 & (t > 0),
  \end{cases}&
  \lim_{\lambda \to -0} \alpha_t(\lambda) &= +\infty, \\
  \lim_{\lambda \to -\infty} \beta_t(\lambda) &= +\infty,&
  \lim_{\lambda \to -0} \beta_t(\lambda) &= 0.
\end{align}
Thus, for every $t \in \Rp$, there exists unique $\lambda \in \R_{<0}$ such that $\alpha_t(\lambda) = \beta_t(\lambda)$, and such $\lambda$ is a unique maximizer of $\min\set{\alpha_t(\lambda), \beta_t(\lambda)}$.
We denote it by $\lambda^*(t)$.
Namely, $\lambda^*(t)$ satisfies
\begin{align}\label{eq:f-t-lambda-t}
  f(t)
  = -\frac{a}{\lambda^*(t)} + \e^{t\lambda^*(t)}
  = -\frac{\lambda^*(t)}{b} + 1 - \e^{t\lambda^*(t)}.
\end{align}
Solving~\eqref{eq:f-t-lambda-t} for $\e^{t\lambda^*(t)}$, we get
\begin{align}\label{eq:exp-t-lamabda}
  \e^{t\lambda^*(t)} = \frac12 \prn{\frac{a}{\lambda^*(t)} - \frac{\lambda^*(t)}{b} + 1}.
\end{align}
Thus, $f(t)$ can also be expressed as
\begin{align}\label{eq:f-t-simple}
  f(t) = \frac12 \prn{-\frac{a}{\lambda^*(t)} - \frac{\lambda^*(t)}{b} + 1}.
\end{align}

Now, $\lambda^*(t)$ is the explicit function obtained by solving an implicit function $g(t, \lambda) \defeq \alpha_t(\lambda) - \beta_t(\lambda) = 0$ for $\lambda$.
The function $g$ is sufficiently smooth on $\R \times \R_{<0}$ and it satisfies $\pdif{g}{\lambda}(t, \lambda) > 0$ for all $(t, \lambda) \in \Rp \times \R_{<0}$.
Therefore, there exists open $U \subseteq \R \times \R_{<0}$ with $\Rp \times \R_{<0} \subseteq U$ such that $\pdif{g}{\lambda}(t, \lambda) \ne 0$ for all $(t, \lambda) \in U$.
By the implicit function theorem, $\lambda^*(t)$ is smooth over $t \in \Rp$ and so is $f(t)$ due to~\eqref{eq:f-t-simple}.
Differentiating~\eqref{eq:f-t-simple} by $t$, we get
\begin{align}
  \dot{f}(t)
  = \frac12 \prn{\frac{a}{{\lambda^*(t)}^2} - \frac{1}{b}} \dot{\lambda}^*(t).
\end{align}
Similarly, differentiating~\eqref{eq:exp-t-lamabda} by $t$, we obtain
\begin{align}\label{eq:dot-f}
  (\lambda^*(t) + t\dot{\lambda}^*(t))\e^{t\lambda^*(t)}
  = \frac12 \prn{-\frac{a\dot{\lambda}^*(t)}{{\lambda^*(t)}^2} - \frac{\dot{\lambda}^*(t)}{b}}
\end{align}
and
\begin{align}\label{eq:lambda-dot}
  \dot{\lambda}^*(t)
  = -\frac{{\lambda^*(t)}^2}{t\lambda^*(t) + \frac12\prn{\frac{a}{\lambda^*(t)} + \frac{\lambda^*(t)}{b}}\e^{-t\lambda^*(t)}}.
\end{align}

Let $t^* \in \Rp$ be a stationary point of $f(t)$, i.e., $\dot{f}(t^*) = 0$.
By~\eqref{eq:dot-f}, $\dot{f}(t^*) = 0$ is equivalent to $\dot{\lambda}^*(t^*) = 0$ or $\frac{a}{{\lambda^*(t^*)}^2} - \frac{1}{b} = 0$.
If $\dot{\lambda}^*(t^*) = 0$, it holds that $\lambda^*(t^*) = 0$ from~\eqref{eq:lambda-dot}, which contradicts $\lambda^*(t^*) < 0$.
Thus, $\frac{a}{{\lambda^*(t^*)}^2} - \frac{1}{b} = 0$ must hold, i.e., $\lambda^*(t^*) = -\sqrt{ab}$.
Substituting this into~\eqref{eq:exp-t-lamabda}, we get
\begin{align}
  \e^{-t^*\sqrt{ab}} = \frac12 \prn{-\frac{a}{\sqrt{ab}} + \frac{\sqrt{ab}}{b} + 1} = \frac12
\end{align}
and thus
\begin{align}
  t^* = \frac{\ln 2}{\sqrt{ab}}.
\end{align}

Finally, by direct calculation, we can check that
\begin{align}
  \dot{\lambda}^*(t^*) &= \frac{ab}{2 \sqrt{\frac{a}{b}} + \ln 2} > 0, \\
  \ddot{f}(t^*)
  &= -\frac{a\dot{\lambda}^*(t^*)}{{\lambda^*(t^*)}^3} + \frac12 \prn{\frac{a}{{\lambda^*(t^*)}^2} - \frac{1}{b}} \ddot{\lambda}^*(t^*)
  = -\frac{a\dot{\lambda}^*(t^*)}{{\lambda^*(t^*)}^3} > 0.
\end{align}
Thus, $t^*$ is a local minimum of $f(t)$, and is indeed the global minimum since it is the unique stationary point of $f(t)$ with $t \in \Rp$.

\section{Steady-states Selection and MAP Inference for DPPs}\label{sec:map-inference-dpp}
The max-diagonal strategy for steady-states selection in \Cref{alg:improved} is closely related to the \emph{maximum a posteriori} (MAP) \emph{inference for determinantal point processes} (DPPs)~\cite{Macchi1975-ps} as follows.

If $S^{(k)}$ could be chosen independently from $S^{(k-1)}$, the steady-states selection boiled down to the task of finding a bipartition $\set{S, T}$ of $[n]$ with $|S| = k$ such that the error of the RCMC method becomes smaller.
Roughly speaking, the results in \Cref{sec:error-analysis} claim that the accuracy of the RCMC method improves when $\sigma(K_{SS})$ is large and $\rho(D)$ with $D = K_{TT} - K_{TS}K_{SS}^{-1}K_{ST}$ is small.
Namely, it is desirable to choose $S \subseteq [n]$ with $|S| = k$ such that all the eigenvalues of $K_{SS}$ and $D$ are large and small, respectively, in terms of the absolute value.

Since $\det K_{SS} \det D = \det K$, using $\abs{\det K_{SS}}$ in place of $\sigma(K_{SS})$ and $\frac{1}{\rho(D)}$, one can model this task as a problem of finding $S \subseteq [n]$ with $|S| = k$ such that $\abs{\det K_{SS}}$ is maximized, or equivalently, $f(S) \defeq \ln {\abs{\det K_{SS}}}$ is the largest.
This problem is known as the \emph{maximum a posteriori} (MAP) \emph{inference for determinantal point processes} (DPPs) if $K$ is symmetric positive (negative) semi-definite~\cite{Chen2018-aa,Macchi1975-ps}.
Since $K$ is written as $-L\Pi^{-1}$ with a symmetric positive semi-definite matrix $L$, it holds that
\begin{align}\label{eq:logdet-KSS}
  f(S) = \ln \det L_{SS} - \sum_{i \in S} \ln \pi_i.
\end{align}
The function $g(S) \defeq \ln \det L_{SS}$ is known to be \emph{submodular}, i.e., $g(i \given X) \ge g(i \given Y)$ holds for all $Y \subseteq [n]$, $X \subseteq Y$, and $i \in [n] \setminus Y$, where $g(i \given S) \defeq g(S \cup \set{i}) - g(S)$ for $S \subseteq [n]$ and $i \in [n]$.
By~\eqref{eq:logdet-KSS}, $f(S)$ is also submodular, and thus the problem of maximizing $f(S)$ is a special case of submodular function maximization under a cardinality constraint.

A standard approach to submodular function maximization with a cardinality constraint is the \emph{greedy algorithm}: starting from $S^{(0)} = \emptyset$, one chooses $j^{(k)} \in \argmax \set[\big]{f\prn[\big]{j \given S^{(k-1)}}}[j \in [n] \setminus S]$ and let $S^{(k)} \defeq S^{(k-1)} \cup \set[\big]{j^{(k)}}$ for $k = 1, 2, \dotsc$.
Indeed, as shown in~\cite{Chen2018-aa}, $f\prn[\big]{j \given S^{(k-1)}}$ coincides with the logarithm of $\abs[\big]{D_{jj}^{(k-1)}}$ by
\begin{align}
  f\prn[\big]{j \given S^{(k-1)}}
  &= \ln {\abs[\big]{\det K_{S^{(k-1)} \cup \set{j}, S^{(k-1)} \cup \set{j}}}} - \ln {\abs{\det K_{S^{(k-1)}S^{(k-1)}}}} \\
  &= \ln {\abs{\frac{\det K_{S^{(k-1)} \cup \set{j}, S^{(k-1)} \cup \set{j}}}{\det K_{S^{(k-1)}S^{(k-1)}}}}} \\
  &= \ln {\abs[\big]{D_{jj}^{(k-1)}}},
\end{align}
meaning that \Cref{alg:improved} is nothing but an implementation of the greedy algorithm for maximizing $f(S)$.
Therefore, the max-diagonal strategy is natural from the viewpoint of error analysis as well.

\section{Incremental Update of LU Decomposition for \texorpdfstring{$M$}{M}}\label{sec:incremental-update-M}
Let $K \in \R^{n \times n}$ be a rate constant matrix and $S = S^{(k)}$ and $T = T^{(k)}$ those defined in \Cref{alg:improved}.
In the RCMC method of Type~B, we solve a linear system with the coefficient matrix
\begin{align}
  M^{(k)}
  \defeq{}& \Omega^{(k)}{\Omega^{(k)}}^*
  = I_T + K_{TS} K_{SS}^{-2}K_{ST},
\end{align}
where
\begin{align}
  \Omega^{(k)} \defeq \begin{pmatrix} -K_{TS} K_{SS}^{-1} & I_{T} \end{pmatrix}.
\end{align}
In the discussions below, we show that one can update a Cholesky factor of $M^{(k-1)}$ into that of $M^{(k)}$ by performing the \emph{rank-one update} of a Cholesky decomposition of a matrix in $\R^{T \times T}$.
Here, the rank-one update of a Cholesky decomposition of a positive definite matrix $A \in \R^{T \times T}$ means a task to compute a Cholesky decomposition of $A + vv^\top$ for given vector $v \in \R^T$ using a Cholesky decomposition of $A$.
Efficient algorithms that run in $\Order\prn[\big]{|T|^2} = \Order(n^2)$ time are available for this problem~\cite{Fletcher1974-bs}.

Let $G^{(k)}$ be a Cholesky factor of $M^{(k)}$ with rows and columns indexed by $T$ and $[|T|] = [n-k]$, respectively.
As with the matrix $C^*$ in \Cref{sec:incremantally-updating-lu-decomposition}, the adjoint matrix ${G^{(k)}}^*$ means ${G^{(k)}}^\top \Pi_{T}^{-1}$.
By technical reasons, we assume that the rows of $G^{(k)}$ are ordered as $j^{(k+1)}, j^{(k+2)}, \dotsc$, i.e., one needs to compute the entire sequence of $j^{(1)}, j^{(2)}, \dotsc$ before calculating the output vectors $p^{(1)}, p^{(2)}, \dotsc$, which is possible since $j^{(k)}$ can be determined without referencing $M^{(k)}$.

Initially, we can take $G^{(0)} = \Pi^{\frac12}$ by $M^{(0)} = I_n$.
Consider the update at the $k$th iteration.
Let $E = E^{(k)} \in \R^{T^{(k)} \times T^{(k-1)}}$ be the matrix defined by~\eqref{def:P}.
By~\eqref{eq:original-omega-lem-3}, $M^{(k)}$ can be written as
\begin{align}\label{eq:M-update}
  M^{(k)}
  = \Omega^{(k)}{\Omega^{(k)}}^*
  = E\Omega^{(k-1)}{\Omega^{(k-1)}}^*{E}^*
  = EM^{(k-1)}E^*
  = EG^{(k-1)}{G^{(k-1)}}^*E^*.
\end{align}
Letting $T = T^{(k)}$, $j = j^{(k)}$, and $J \defeq \set{2, \dotsc, n-k}$, we have
\begin{align}\label{eq:EG}
  EG^{(k-1)}
  = \begin{pmatrix}
    -\frac{D^{(k-1)}_{Tj}}{D^{(k-1)}_{jj}} & I_T
  \end{pmatrix}
  \begin{pmatrix}
    G_{j1}^{(k-1)} & \zerostr \\
    G_{T1}^{(k-1)} & G_{TJ}^{(k-1)}
  \end{pmatrix}
  = \begin{pmatrix}
    v & G_{TJ}^{(k-1)}
  \end{pmatrix},
\end{align}
where
\begin{align}
  v \defeq G_{T1}^{(k-1)} - \frac{G_{j1}^{(k-1)}}{D^{(k-1)}_{jj}}D^{(k-1)}_{Tj} \in \R^T.
\end{align}
By~\eqref{eq:M-update} and~\eqref{eq:EG}, it holds that
\begin{align}\label{eq:M-chol-update}
  M^{(k)} = \begin{pmatrix} v & G_{TJ}^{(k-1)} \end{pmatrix}
  \begin{pmatrix} v^* \\ {G_{TJ}^{(k-1)}}^* \end{pmatrix}
  = G_{TJ}^{(k-1)} {G_{TJ}^{(k-1)}}^* + vv^*
\end{align}
with $v^* \defeq v^\top\Pi_T^{-1}$.
The equation~\eqref{eq:M-chol-update} means that $G^{(k)}$ is obtained as the rank-one update of $G_{TJ}^{(k-1)} {G_{TJ}^{(k-1)}}^*$ for $v$.
Analyzing numerical stability on the rank-one update is left for further investigation.

\end{document}